\documentclass{amsart}

\usepackage[T1]{fontenc} 
\usepackage[utf8]{inputenc}
\usepackage[english]{babel} 

\usepackage{amsmath,amssymb}
\usepackage{thmtools} %to make cleverref work!

\usepackage[dvipsnames]{xcolor}
\usepackage{amsthm}
\usepackage[pagebackref, colorlinks=true, linkcolor=MidnightBlue, citecolor=OliveGreen]{hyperref}
\usepackage[capitalise,nameinlink]{cleveref}
\usepackage{mathdots}

\usepackage{url}
\usepackage{booktabs}
\usepackage{lipsum} 

\usepackage{ dsfont }

\usepackage{amsfonts,amssymb,mathtools}
\usepackage{dsfont}% per scrivere dei carateeri matematici
\usepackage{subfig} %per fare delle tabelle
\usepackage[all]{xy} %per scrivere facilmente delle matrici

\usepackage{mathrsfs}
\usepackage{amscd}

\usepackage{extpfeil}

\DeclareMathOperator{\Aut}{Aut}
\DeclareMathOperator{\rk}{rk}

\DeclareMathOperator{\diam}{diam}

\newcommand{\Zp}{\mathbb{Z}_p}

\theoremstyle{plain}

\newtheorem{theorem}{Theorem}[section]
\newtheorem{lemma}[theorem]{Lemma}
\newtheorem{proposition}[theorem]{Proposition}
\newtheorem{corollary}[theorem]{Corollary}

\newtheorem*{question}{Question}
\newtheorem{remark}[theorem]{Remark}

\newtheorem*{MainQuestion}{Main Question}

\newtheorem{thmABC}{Theorem}

\newtheorem{propABC}[thmABC]{Proposition}

\theoremstyle{definition}
\newtheorem{definition}[theorem]{Definition}
\newtheorem{example}[theorem]{Example}

 \usepackage{todonotes}

\usepackage{quoting}
\quotingsetup{font=small}

\usepackage{graphicx} %per inserire le figure
\usepackage{chngcntr}% questi non so bene ma male non fanno lol
\usepackage{ifthen}
\usepackage{calc}

\newcommand{\commatteo}[1]{\footnotemark\marginpar{\color{blue} \tiny \footnotemark[\arabic{footnote}] #1}}

\newcommand{\commandrea}[1]{\footnotemark\marginpar{\color{violet} \tiny \footnotemark[\arabic{footnote}] #1}}

\newcommand{\commargherita}[1]{\footnotemark\marginpar{\color{teal} \tiny \footnotemark[\arabic{footnote}] #1}}

\allowdisplaybreaks

\usepackage[capitalise, nameinlink]{cleveref}
\usepackage{csquotes}
\usepackage[normalem]{ulem}

\usepackage{bbm}

\usepackage{enumitem}

\newcommand{\comm}[1]{}

\DeclareMathOperator{\Dv}{\Delta_{\mathrm{virt}}}

\DeclareMathOperator{\Gv}{\Gamma_{\mathrm{virt}}}

\thanks{The third author is a member of GNSAGA (INdAM) and kindly acknowledges their support.
 The fourth author is funded by the Italian
program Rita Levi Montalcini for young researchers (Ed.\ 2021). %{\color{red}We would like to thank J.\ Moritz Petschik for spotting a mistake in an previous version of the article.}}}
}

\title[Virtually generating graphs of pro-$p$ groups]{Virtually generating graphs of pro-$p$ groups}

\author{Yiftach Barnea}
\address{Royal Holloway, University of London, Egham, UK}
\email{y.barnea@rhul.ac.uk}

\author{Andrea Lucchini}
\address{Universit\`a degli Studi di Padova, Padova, Italy}
\email{lucchini@math.unipd.it}

  \author{Margherita Piccolo}
\address{FernUniversit\"at in Hagen, Hagen, Germany} \email{margherita.piccolo@fernuni-hagen.de}

    \author{Matteo Vannacci}
\address{Università degli Studi di Firenze, Florence, Italy}
\email{matteo.vannacci@unifi.it}



\colorlet{red}{black}

\begin{document}

\begin{abstract}We study the virtually generating graph of pro-$p$ groups. We show that various classes of pro-$p$ groups have connected virtually generating graph and we bound its diameter in these cases; e.g.\ compact subgroups of analytic groups over local fields and the Nottingham group.\end{abstract}

\maketitle

%\commatteo{Title: Virtually generating graphs and openisers of pro-$p$ groups}

\section{Introduction}
The \textit{virtually generating graph} $\Gv(G)$ of a topological group~$G$ is the graph with vertex set $G$ and where two elements are adjacent if and only if they topologically generate an open subgroup of~$G$. The \textit{proper} virtually generating graph is the subgraph $\Dv(G)$ of $\Gv(G)$ obtained by removing its isolated vertices.

The virtually generating graph was suggested to the second author by the first author as a variant of the \textit{generating graph} $\Gamma(G)$ of a group $G$, that is, the graph whose vertex set is $G$ and where two elements are adjacent if and only if they topologically generate the entire group~$G$. We write $\Delta(G)$ for the subgraph of $\Gamma(G)$ obtained by removing its isolated vertices. The graph $\Delta(G)$ has received a lot of attention in the last years. For a survey of recent results in the area see, for instance, \cite{Bu} or \cite{Sc}. 

It is conjectured that $\Delta(G)$ is connected for any finite group $G$. However, this is not always the case for infinite profinite groups, as there exists a $2$-generated profinite group~$G$ such that both~${\Delta(G)}$ and~${\Dv(G)}$ have~$2^{\aleph_0}$ connected components, see~\cite{Luc20,Luc21}. One might think that the connectedness of the graphs $\Delta(G)$ and $\Dv(G)$ behaves similarly, but this is not the case for prosoluble groups. Indeed, in~\cite{Luc21} the second author shows that for every positive integer $t$ there exists a finitely generated prosoluble group $G$ such that the graph ${\Dv(G)}$ has precisely $t$ connected components, while the graph ${\Delta(G)}$ is connected of diameter at most~$3$ for every prosoluble group, cf.\ \cite{Luc20}. 

The next natural class to investigate is the class of pronilpotent groups. This was initiated in \cite{Mor22}.   
Since pronilpotent groups are cartesian products of their Sylow pro-$p$ subgroups, the investigation on the connectedness of ${\Dv(G)}$ reduces to the pro-$p$ case, see \cite[Proposition 1.3.3]{Mor22}. In \cite{Mor22} the virtually generating graph of several examples was studied (abelian pro-$p$ groups, free pro-$p$ groups, and the Nottingham group $\mathcal{N}(\mathbb{F}_p)$ for $p\geq5$), proving connectedness of the graph~$\Dv$ in these cases. 

In this paper we address the following question. By convention we consider the empty graph to be connected.
\begin{MainQuestion} Let $G$ be a pro-$p$ group. Is the proper virtually generating graph $\Dv(G)$ always connected? If so, is its diameter universally bounded?
\end{MainQuestion}

In this article, we will show that the answer to this question is positive for several interesting classes of pro-$p$ groups. To do so, it is useful to introduce a set that measures how `easy' it is to generate an open subgroup together with a fixed element. Let $G$ be a profinite group and $x\in G$. The \emph{openizer} of $x$ is  $$\mathbb{O}_x = \{y\in G  \mid \langle x,y\rangle \le_o G\}.$$
The idea in the background for many results of this paper is the following: if the openizer of a random element is `large' with high probability in a pro-$p$ group $G$, then
%is that having `large' openizers for a `large' subset of elements of a pro-$p$ group $G$ gives us 
we obtain connectivity of $\Dv(G)$ and some bounds on the diameter; see Lemma~\ref{lem:big_open} for the statement of our main technical lemma. This idea together with the results from \cite{BL04} readily yield the following which will be proved in \cref{sec:proof_thm_A}.

\begin{thmABC}\label{prop:A}
    %\item 
    Let $G$ {\color{red}be} a compact subgroup of $\mathcal{G}(F)$, where $\mathcal{G}$ is a semisimple algebraic group and $F$ is a non-archimedean local field of characteristic $p\ge 0$. Assume that $p$ does not divide the degree of the universal covering map $\mathcal{G}^{\mathrm{sc}}\to \mathcal{G}$, then ${\diam(\Dv(G))\leq2}$.
%\end{enumerate}
\end{thmABC}

In \cref{sec:p-adic-analytic}, we extend \cref{prop:A} to uniform pro-$p$ groups which are not necessarily semisimple. 
We recall that a profinite group $G$ has \emph{lower rank} at most $ r$, written $\mathrm{lr}(G)\le r$, if for all $H\le_o G$ there exists $L\le_o G$ with $L\le H$ and $\mathrm{d}(L)\le r$. By \cite{Ku51}, all compact subgroups of semisimple groups have lower rank at most $2$. 

{\color{red}\begin{thmABC}\label{thm:A} 
%The following (virtually) pro-$p$ groups have the graph $\Dv$ connected and with bounded diameter.
%\begin{enumerate}
    %\item 
    Let $U$ be a uniform pro-$p$ group with $\mathrm{lr}(G)\le 2$. Then, $\diam(\Dv(U))\leq2$.
\end{thmABC}}

%\begin{propABC}\label{thm:A} 
%The following (virtually) pro-$p$ groups have the graph $\Dv$ connected and with bounded diameter.
%\begin{enumerate}
    %\item 
%    Let $G$ be a $p$-adic analytic group with lower rank at most $2$ and Haar measure of torsion elements equal to zero. {\color{red}Suppose additionally that, for a uniform subgroup $U$ of $G$, we have that $\Dv(U)=U\smallsetminus\{1\}$.} Then, $\diam(\Dv(G))\leq4$.
%\end{propABC}

%By a recent result of Toti \cite{Tot22}{\color{red} and by \cref{prop:uniform_full_Dv}
%{\color{red}In \cref{rmk:hypotheses}, we will show that the hypotheses of \cref{thm:A} hold for any compact subgroup of the $\mathbb{Q}_p$-points of a simple algebraic group. Moreover, in \cref{lem:metabelian_examples} we will exhibit a family of metabelian pro-$p$ groups that fit the hypotheses of \cref{thm:A}.} %$p$-adic analytic group which is not virtually solvable has torsion of Haar measure zero{\color{red}. Moreover, we will prove below that a uniform subgroup $U$ of $\mathcal{G}(\mathbb{Q}_p)$ H}ence, every $p$-adic analytic group with lower rank at most $2$ which is not virtually solvable satisfies the hypotheses of \cref{thm:A}.

%On the other hand, 
It seems to be a very difficult task to calculate the exact diameter for the proper virtually generating graph of %a solvable 
{\color{red}general} $p$-adic analytic groups with torsion {\color{red}which are not semisimple}. {\color{red}In \cref{sec:improvement} we extend the proof of \cref{thm:A} and we show that certain $p$-adic analytic groups with torsion set of Haar-measure zero have connected proper virtually generating graph of diameter at most $4$.

Continuing the study of {\color{red}non-semisimple} $p$-adic analytic pro-$p$ groups with torsion,} we give three examples of virtually-abelian $p$-adic analytic groups for which we can draw conclusions on the connectivity and the diameter.

\begin{propABC}\label{propABC:virt_ab} \quad 
\begin{enumerate}%[leftmargin=1.5em]
    \item The proper virtually generating graph $\Dv$ of the pro-$p$ group of maximal class has diameter $2$. %$\diam(\Dv(\mathbb{Z}_p[\zeta]\rtimes C_p))=2$.
    \item Let $G$ be the semidirect product $V \rtimes Q_8$ where $V$ is a free $\mathbb{Z}_p$-module of rank~4 with basis $\{\mathds{1},i,j,k\}$ and $Q_8$ is the quaternion group acting by multiplication. Then, $\diam(\Dv(G))=2$.
    \item The proper virtually generating graph $\Dv$ of the wreath product $\mathbb{Z}_2 \wr D_4$ has diameter~$3$. %For the wreath product $\mathbb{Z}_2 \wr D_4$, then $\diam(\Dv(\mathbb{Z}_2 \wr D_4))=3$.
\end{enumerate}
\end{propABC}

The parts of \cref{propABC:virt_ab} are proved in \cref{Sec:maxClass}, \cref{ex:Q8} and \cref{sec:small}, respectively. We remark that the wreath product in part (3) of \cref{propABC:virt_ab} is the only pro-$p$ group for which we can show that the graph $\Dv$ is of diameter $3$. We believe the virtually generating graphs of virtually abelian pro-$p$ groups should be studied further.

Up to this point most groups that we considered were of finite rank. For this reason, we also consider the wreath product $C_p \wr \mathbb{Z}_p$ and we show that $\diam(\Dv(C_p\wr \mathbb{Z}_p))=2$, see \cref{sec:WreathProcuts}.

Continuing with groups of infinite rank, using again \cref{lem:big_open}, we can study the connectivity of the virtually generating graph for the Nottingham group and some of its subgroups $\mathcal{Q}^1(s,r)$ defined by Ershov in \cite{Ers04}.

\begin{thmABC}\label{thm:nott_intro}
Let $p\ge 5$ be a prime and let $\mathcal{N}=\mathcal{N}(\mathbb{F}_p)$ be the Nottingham group. Then,  we have $\diam(\Dv(\mathcal{N}))=2$.
Moreover, if $\mathcal{Q}^1(s,r)$ is the subgroup of $\mathcal{N}$ defined in \cite{Ers04}, then $\diam(\Dv(\mathcal{Q}^1(s,r)))= 2$.
\end{thmABC}

We note that the proofs for the Nottigham group and those for its subgroups are quite different; cfr.\  \cref{sec:Nottingham} and \cref{sec:Lie-algebra-methods}, respectively. In particular, in the proof of \cref{thm:nott_intro} we sharpen Mori's bound on the diameter of the virtually generating graph of the Nottingham group, proving lower bounds on the measure of the openizers of all elements.

 As mention before, many of the aforementioned results rely on having `large' openizers and it is easy to see that non-empty openizers in pro-$p$ groups always have positive measure, see \cref{lem:open_pos_measure}. In \cref{sec:openizers_profinite} we will show that openizers in general profinite groups do not need to be neither closed nor of positive measure. At the same time, we will prove that openizers in pro-$p$ groups can have arbitrarily small positive measure.

\begin{propABC}
    For any $\varepsilon >0$ there exists a pro-$p$ group $G_\varepsilon$ and an element $x\in G_\varepsilon$ such that $0<\mu(\mathbb{O}_x) <\varepsilon$.
\end{propABC}

The previous proposition is proved in \cref{prop:small}.
We believe that it would be interesting to further study openizers in pro-$p$ groups. For instance, to best of our knowledge the following is open.

\begin{question}
    Is there a pro-$p$ group $G$ such that there exists a sequence of elements $x_n\in G$ with $0<\mu(\mathbb{O}_{x_n})<\frac{1}{n}$?
\end{question}

\comm{
\subsection*{Organisation of the paper} 
After proving some generalities about the openizer of an element in profinite groups, we establish part (1) of \cref{thm:A} which is referred as \cref{cor:barnea_larsen}. We then produce some examples where openizers have exotic properties (e.g.\ not closed and dense). In \cref{sec:p-adic-analytic}, we deal with $p$-adic analytic groups. Here we obtain proofs of parts $(2)$--$(4)$. In \cref{sec:Lie-algebra-methods} we prove part $(8)$. The Nottingham groups are treated in \cref{sec:Nottingham}. Part (6) is established in \cref{sec:WreathProcuts}. In the last section we discuss on pro-$p$ groups with small openizers and we prove part $(5)$.
}

\subsection*{Notation}
We denote by $\mathbb{N}$ the set of positive integers. As it is customary when working with profinite groups, unless stated otherwise, we will assume that all subgroups are closed, all homomorphism are continuous and generation will be intended in the topological sense. 

The word `measure' will always refer to the Haar measure in the relevant locally compact group. On a profinite group $G$, we will always consider the normalised Haar measure $\mu$ such that $\mu(G)=1$.

\subsection*{Acknowledgments}
{\color{red}We would like to thank J.\ Moritz Petschik for spotting a mistake in an previous version of the article.}

\section{Openizers in profinite groups}\label{sec:openizers}

\comm{
Let $G$ be a profinite group. If ${\Dv(G)}$ is not the empty graph, then $G$ is finitely generated. In fact, if $\langle x,y\rangle\leqslant_o G$, then there exists $N\trianglelefteq_o G$ such that $N\subseteq \langle x,y\rangle$. Then there exist $d\in \mathbb{N}$ and $x_1,\dots,x_d$ elements of $G$ such that $x_1N,\dots,x_dN$ generate~$G/N$ which is finite. Thus $G=\langle x_1,\dots, x_d,x,y\rangle$, cf.\ \cite[Section 1.1]{Mor22}.\commargherita{This is probably out of place} 
}

\subsection{Basics on openizers}

\begin{lemma}For every $g\in G$, the subset $\mathbb O_g$ is measurable.
\end{lemma}
\begin{proof}Let $\mathcal H$ be the set of the open subgroups $H$ of $G$ containing the element $g$. For $H\in \mathcal H,$ let $\Omega_H(g)=\{h\in H\mid \langle h,g\rangle =H\}.$ Notice that $H\setminus \Omega_H(g)$ is the union of the maximal open subgroup of $H$ containing $g$. Hence $H \setminus \Omega_H(G)$ is open in $H$ and consequently $\Omega_H(G)$ is closed. Hence
$\mathbb O_g=\cup_{H\in \mathcal H}\Omega_H(g)$ is a union of closed subsets, and therefore it is measurable.
\end{proof}

\comm{

\begin{lemma}\label{lem:open_pos_measure}
    The measure of the openizer of a non-isolated element in a pro-$p$ group is always positive.
\end{lemma}
\commandrea{We have a similar statement if $G$ is virtually pro-$p$; we consider the Frattini subgroup $F$ of $O_p(H)$ and notice that the openizer of $g$ contains a coset of $F$}
\begin{proof}
Let $G$ be a pro-$p$ group and $g\in G$. If $\mathbb O_g\neq \emptyset$, then there exists $x\in G$ such that $\langle g,x\rangle$ is an open subgroup of $G$. Put $H=\langle g,x\rangle$. {\color{blue}Working modulo the Frattini subgroup of $H$, it is easy to see that} the probability that $\langle g, h\rangle=H$ for $h\in H$ is $1-\frac{1}{p}$. It follows that
$$\mu(\mathbb O_g)\geq \left(1-\frac{1}{p}\right)\frac{1}{|G:H|}. \qedhere$$ 
\end{proof}

}

\begin{lemma}\label{lem:open_pos_measure}
The openizer of a non-isolated element in a virtually pro-$p$ group is always open, and therefore has positive measure.
\end{lemma}
\begin{proof}
Let $G$ be a virtually pro-$p$ group and $g\in G$. If $\mathbb O_g\neq \emptyset$, then there exists $x\in G$ such that $\langle g,x\rangle$ is an open subgroup of $G$. Put $H=\langle g,x\rangle$. Since $G$ is virtually pro-$p$, $H$ is also virtually pro-$p$. In particular, $K=O_p(H)$ is open in $G$ and finitely generated. Thus the Frattini subgroup $F$ of $K$ is open in $G$. Since $F$ is contained in the Frattini subgroup of $G$, $xF\subseteq \mathbb O_g$. So $\mathbb O_g$ contains an open neighborhood of each of its elements.
\end{proof}

We will show in~\cref{sec:openizers_profinite} that the previous lemma is false for profinite groups and that it is not necessarily true that the openizer of an element is open.

\begin{lemma}\label{lem:big_open} Let $G$ be a pro-$p$ group. 
    \begin{enumerate}
        \item If the set $ B= \{x\in G \mid \mu(\mathbb{O}_x)> 1/2\}$ has measure one, then $\Dv(G)$ is connected and it has diameter at most $4$. Moreover, if $B=G\smallsetminus\{1\}$, then $\Gv(G)\smallsetminus \{1\}$ has diameter at most $2$.
        \item If there is an element $x\in G$ such that $\mu(\mathbb{O}_x)=1$, then $\Dv(G)$ is connected and of diameter at most $2$. 
    \end{enumerate}
\end{lemma}
\begin{proof}
(1) For any two elements $x,y\in B$ it is clear that $\mathbb{O}_x \cap \mathbb{O}_y \neq \emptyset$; hence, there exists $z\in G$ connected to both $x$ and $y$. So any two elements in $B$ are at distance at most 2 in $\Dv(G)$.

Consider now a non-isolated element $x \in G$ with $x\notin B$. By~\cref{lem:open_pos_measure}, $\mu(\mathbb{O}_x)>0$, so that $B\cap \mathbb{O}_x \neq \emptyset$ and we can connect $x$ to an element of $B$. It follows that the diameter of $\Dv(G)$ is at most $4$.

(2) By~\cref{lem:open_pos_measure}, we have that $\mathbb{O}_x \cap \mathbb{O}_y \neq \emptyset$ for every $y\in G$. The statement follows with the same argument of the first part. 
\end{proof}

\subsection{Openizers and probability of generating an open subgroup}\label{sec:proof_thm_A}

Let $Q(G,k)$ the probability that $k$ random elements generate an open subgroup in~$G$. The first to introduce the probability $Q(G, k)$ was Avinoam Mann in \cite[pg. 435]{Man96} and it has been extensively studied for profinite groups. We will obtain an equivalent condition for $Q(G,2)=1$ using openizers. 

Let $G$ be a profinite group. Let $L_t=\{ x \in G \mid  \mu(\mathbb{O}_x) \ge t \}$ %Let $C_t= G\smallsetminus L_t$ be its complement.
and notice that $L_1=\bigcap_{n\in \mathbb N}  (L_{\frac{n-1}{n}})$.
%look at Barnea and Larsen

\begin{lemma}
    Let $G$ be a finitely generated profinite group. Then,  for any $t\in [0,1]$, the set $L_t \subseteq G$ is measurable.
\end{lemma}
\begin{proof}
Let $Y=\{(g,h)\in G\times G \mid \langle g,h\rangle \le_o G\}$. 
\comm{
Since $G$ has only countably many open subgroups we can write 
    \[
     Y=\bigcup_{H\in \mathcal{S}_G} \{(g,h)\in G\times G \mid \langle g,h\rangle =H\} 
    \]
    where $\mathcal{S}_G$ is the set of open subgroups of $G$.

Now, it is easy to see that
\[
  \{(g,h)\in G\times G \mid \langle g,h\rangle =H\} = (H\times H) \smallsetminus \bigcup_{K\in \mathcal{M}_H} M\times M
\]
where $\mathcal{M}_H$ is the set of maximal open subgroups of $H$. Since $\mathcal{M}_H$ is again countable for every $H$, we can conclude that $Y$ is measurable. (Up to here is contained in the book "Subgroup Growth")
}%
 It is standard to see that $Y$ is a closed subset in $G\times G$, see for instance \cite[Chapter~11]{LS03}.
Since $Y$ is measurable, its indicator function $ {1}_Y(g,h)$ is a measurable function and, by Fubini's theorem (\cite[Thm.~8.8(a)]{Rud87}), the function 
\[
  f(g) = \int_G  {1}_Y(g,h) \ d\mu(h) 
\]
is measurable. Note that $1_Y(g,h)=1$ exactly when $h\in \mathbb{O}_g$; hence, $f(g) = \mu(\mathbb O_g)$.

Finally, we have that $L_t = \{g\in G \mid f(g) \ge t\} = f^{-1}([t,1])$ is measurable. 
\end{proof}

\begin{proposition}\label{prop:Q_L_1}
  Let $G$ be a profinite group. Then, $Q(G,2)=1$ if and only if $\mu(L_1)=1$.  
\end{proposition}
\begin{proof}
If $\mu(L_1)=1$, then we pick two random elements $x$ and $y$. With probability $1$ the first element $x$ is in $L_1$ and then with probability $1$ the second element $y$ is in $\mathbb{O}_{x}$. Thus, $Q(G,2)=1$.

Now suppose $Q(G,2)=1$. Assume by contradiction that $\mu(L_1)<1$. Since $L_1=\bigcap_{n\in \mathbb N}  L_{\frac{n-1}{n}}$, there exists $n$ such that $\mu(L_{\frac{n-1}{n}})=1-\varepsilon$ for some positive $\varepsilon$. It follows that, for a random element $x\in G$, with probability $\varepsilon$ we have $x \notin L_{\frac{n-1}{n}}$. Then, with probability greater than $1/n$ a random element $y$ is not in $\mathbb{O}_x$. Thus, with positive probability two random elements $x,y$ do not generate an open subgroup, in contradiction to $Q(G,2)=1$.
\end{proof}

{\color{red}
Now \cref{prop:A} is an immediate consequence of \cite[Thm.~4.5]{BL04}, \cref{prop:Q_L_1} and \cref{lem:big_open} part (2). We restate \cite[Thm.~4.5]{BL04} below for the convenience of the reader.

\begin{theorem}[{\cite[Thm.~4.5]{BL04}}]
 Let $\mathcal{G}$ be a semisimple algebraic group over a field $K$ in characteristic $p\ge0$. Let $G$ denote a compact open subgroup of $\mathcal{G}(K)$. Then $Q(G, 2) = 1$ unless $p$ divides $\lvert Z\rvert$ in which case $G$ is not finitely generated.
\end{theorem}

%\begin{corollary}\label{cor:barnea_larsen}
%  Let $\mathcal{G}$ be a semisimple algebraic group over a field $K$ of characteristic~$p$. Let $G$ denote a compact open subgroup of $\mathcal{G}(K)$ and suppose that $G$ is finitely generated. Then $\Dv(G)$ is connected and of diameter at most $2$.
%\end{corollary}
}

%%%%%%%%%%%%%%%%%%%%%%%%%%%%%%%%%%%%%%%%%%%%%%%%%%%%%%

\section{Examples of exotic openizers} \label{sec:openizers_profinite}

In this section we examine the behaviour of the measure of openizers in profinite groups. We will exhibit examples that show that the openizers in profinite groups can be dense and of measure zero, even for solvable profinite groups. Moreover, we will show that openizers do not need to be closed in pro-$p$ groups.

\subsection{Openizers in cartesian products of simple groups}
We begin with some number-theoretic lemmas.
\begin{lemma}

\[
	T(n) = \sum_{\substack{a+b=n \\ a,b \geq 1}} \frac{1}{ab} 
	= \sum_{a=1}^{n-1} \frac{1}{a(n-a)}\leq \frac{2\ln n + 1}{n} .
	\]
\end{lemma}
\begin{proof}	
Since
	\[
	\frac{1}{a(n-a)} = \frac{1}{n}\left(\frac{1}{a} + \frac{1}{n-a}\right),
	\]
it follows that
	\[
	T(n) 
	= \frac{1}{n} \sum_{a=1}^{n-1} \left(\frac{1}{a} + \frac{1}{n-a}\right)
	= \frac{2}{n} \sum_{a=1}^{n-1} \frac{1}{a}
	= \frac{2\,H_{n-1}}{n},
	\]
where $H_k = \sum_{j=1}^{k} \frac{1}{j}$ is the $k$-th harmonic number. We conclude using the classical estimation 
 $H_k \leq \ln k + 1$ for all $k \geq 1$.
\end{proof}

\begin{lemma}
	\[
	S(n) = \sum_{\substack{a+b+c=n \\ a,b,c \geq 1}} \frac{1}{abc}\leq \frac{4(\ln n + 1)^2}{n}.
	\]
\end{lemma}	
\begin{proof}
	\[
	S(n) = \sum_{a=1}^{n-2} \frac{1}{a} \cdot T(n-a)
	= \sum_{a=1}^{n-2} \frac{1}{a} \cdot \frac{2\,H_{n-a-1}}{n-a}= \frac{4}{n} \sum_{a=1}^{n-2} \frac{H_{n-a-1}}{a}.
	\]
Since $H_{n-a-1} \leq \ln n + 1$ for all $a=1,\ldots,n-2$:
	\[
	S(n) 
	\leq \frac{4(\ln n+1)}{n} \sum_{a=1}^{n-2} \frac{1}{a}
	= \frac{4(\ln n+1)}{n} \cdot H_{n-2}
	\leq \frac{4(\ln n+1)^2}{n}. \qedhere
	\]
\end{proof}
\begin{lemma}\label{stime}
Let $p_n$ be the probability that a randomly chosen element $x \in A_n$ satisfies $\langle x, (1,2,3)\rangle=A_n$.
Then $$p_n\leq \frac{1+2(\ln n+1)+4(\ln n+1)^2}{n}.$$
\end{lemma}
\begin{proof}If $\langle x, (1,2,3)\rangle=A_n$ then in particular  $\langle x, (1,2,3)\rangle$ is transitive, so
	$x$ is the product of at most 3 disjoint cycles, and therefore the number of choices for $x$ is at most
	$$n! \left(\frac{1}{n}+T(n)+S(n)\right).\qedhere$$
    \end{proof}	

\begin{proposition}Let $G=\prod_{n\geq 5}A_n$. If $g=(g_n)_{n\ge 5} \in G$, then $\mathbb O_g \neq \emptyset$ if and only if there are only finitely many $n\geq 5$ with $g_n=1$. Moreover if $\mathbb O_g \neq \emptyset,$ then it is dense in $G$.
\end{proposition}
\begin{proof}
Clearly $(y_n)_{n\ge 5}\in \mathbb O_g$ if and only if  for all $n$, except at most a finite number, $\langle g_n, y_n\rangle=A_n.$ Since one can find $y_n$ such that $\langle g_n, y_n\rangle = A_n$ if and only if $g_n \neq  1,$ it follows immediately that $\mathbb O_g \neq\emptyset,$ if and only if $g_n \neq 1$ for all but finitely many $n$. 

Let now $t\geq 5$ and pick $y \in \mathbb O_g$. Write $G=Q_t \times N_t$ where $\pi_t:G \to Q_t = \prod_{5\le n \le t} A_n$ is the natural projection and $N_t =\ker \pi_t$. For all $x\in Q_t$, we have that $x y\in \mathbb O_g$; because multiplying by an element of $Q_t$ only changes the first $t$ coordinates. Hence, $N_t \mathbb O_g =G$ and $\mathbb{O}_g$ is dense in $G$.
%
%$N_t=\ker \pi_5\cap\dots \cap \ker \pi_t.$ If $(y_5,\dots )\in \mathbb O_g$, then $(y_5+x_5,\dots,y_t+x_t,y_{t+1},\dots) \in \mathbb O_g$ for every $(x_5,\dots,x_t)$ and therefore, if $\mathbb O_g\neq \emptyset$, then $N_t\mathbb O_g=G.$
\end{proof}

We will need the following classical result from probability theory, known as the First Borel-Cantelli lemma.

\begin{theorem}[{\cite[pp.~389-392]{Ren70}}]\label{bc}
Let $\Omega$ be a probability space with probability $p$. $(X_n)_{n\in \mathbb N}$ be a sequence of events in $\Omega$ such that $\sum_{n\in \mathbb N} p(X_n)< \infty$.
Then, almost surely, only finitely many $X_n$'s will occur.
\end{theorem}

\begin{theorem}
Let $G=\prod_{n\geq 5}A_n$ and consider $x=(x_n)_{n\ge 5}\in G$, with $x_n=(1,2,3)$ for every $n\geq 5.$ Then $\mathbb O_x$ is dense in $G$ and $\mu(\mathbb O_x)=0.$
\end{theorem}
\begin{proof} For $m$ odd, $m\geq 5,$ let $G_m=A_m\times A_{m+1}$ and let
	$$X_m=\{g=(g_n)_n\in G\mid\langle (g_{m},g_{m+1}),((1,2,3),(1,2,3))\rangle=G_m\}.$$
Recall that $p_n$ denotes the probability that a randomly chosen element $x \in A_n$ satisfies
$\langle x, (1, 2, 3)\rangle = A_n$.
Notice that $\mu(X_m) \le p_mp_{m+1}$, so by \cref{stime},
$\sum_m \mu(X_m)$ is convergent. Hence, by the First Borel-Cantelli Lemma, the probability that infinitely many of the events $X_m$ occur is 0, and this implies $\mu(\mathbb O_x)=0.$
\end{proof}

\subsection{Openizers in procyclic profinite groups}
A variation of the previous argument can also be applied applied to the procyclic  profinite group $G = \prod_{p \text{ prime}} C_p$.
Let $(p_n)_{n\in \mathbb N}$ be the sequence of prime numbers in increasing order. Since
$$\prod_n\left(1-\frac{1}{p_n}\right)=0,$$
there is a sequence $(r_m)_{m\in \mathbb N}$ of positive integers with the property that
$$\gamma_m:=\prod_{r_m<k\leq r_{m+1}}\left(1-\frac{1}{p_k}\right)\leq \frac{1}{m^2}.$$ Let 
$H_m=\prod_{r_m<k\leq r_{m+1}}C_{p_k}. $
Clearly $G=\prod_{m\in\mathbb N} H_m$ and $$\mathbb O_1=\{(h_m)_{m\in \mathbb N}\mid
\langle h_m\rangle=H_m \text { for all but finitely many $m$}\}.$$ Let $X_m=\{(h_m)_{m\in \mathbb N}\mid
\langle h_m\rangle=H_m\}.$ Since $\mu(X_m)=\gamma_m,$ the infinite sum $\sum_{m\in \mathbb N}\mu(X_m)$ is convergent, and it follows from the First Borel-Cantelli lemma that $\mu(\mathbb O_1)=0.$

\subsection{Openizers in profinite groups involving few primes}
The previous subsections might led us to believe that we need infinitely many primes dividing the order of our profinite group to construct dense openizers. Here we show that this is not necessary, as we can construct metabelian profinite groups whose order is divisible by exactly two primes with dense openizers.

Let $p$ and $q$ be two different primes. For every $n\in \mathbb N,$ let $a_n$ be the order of $q$ modulo $p^n.$
Let $V_n$ be the additive group of $\mathbb F_{q^{a_n}}$ and
	$x_n$ be an element of order $p^n$ in the multiplicative group of this field. We define an action of $\mathbb Z_p=\langle z\rangle$ on $V=\prod_{n\in \mathbb{N}} V_n$ by setting $v^z=x_nv$ for every $v\in V_n$ and $n\in \mathbb{N}$. Let $G=V\rtimes \mathbb Z_p.$ Then $(v_n)_{n\in \mathbb N}\in V\cap \mathbb O_z$ if and only if $v_n\neq0$ for only finitely many $n\in \mathbb N.$
	This implies that $\mathbb O_z$  is not a closed subset of $G$. 
    
    We will now show that $\mathbb{O}_z$ is dense in $G$. Fix $t\in \mathbb Z_p$ and let $\Omega_t=Vt \cap \mathbb O_z$. Then $\Omega_t t^{-1}$ contains the direct sum of the $V_n$, so $\overline{\Omega_t t^{-1}}=V$ and therefore $Vt$ is contained in the closure of $\mathbb O_z$. We conclude that $\mathbb{O}_z$ is dense in $G$, since $G=\bigcup_{t\in \mathbb{Z}_p} Vt$. %\commandrea{$Vt$ is contained in the closure for every $t$, so every element of $G$ is in the closure}

\subsection{Openizers in pro-\texorpdfstring{$p$}{p} groups}
As we noted in~\cref{lem:open_pos_measure}, openizers in pro-$p$ groups always have positive Haar measure. However, openizers do not need to be closed subsets. Here we just note that this is a consequence of~%\cref{cor:graph_uniform} 
\cref{prop:A} {\color{red}and \cref{prop:Q_L_1}: many} openizers in $\mathrm{SL}_2^1(\mathbb{Z}_p)$ are of measure one and they are not closed. In fact, if the openizer of one element were closed, its complement would be open and, hence, of positive measure.

%%%%%%%%%%%%%%%%%%%%%%%%%%%%%%%%%%%%%%%%%%%%%%%%%%%%%%%%%%

\section{\texorpdfstring{$p$}{p}-adic analytic groups}
\label{sec:p-adic-analytic}

%Throughout this section $G$ will be a $p$-adic analytic group, unless.
In this section, we deal with compact $p$-adic analytic groups. It is well known that a profinite group $G$ admits the structure of a $p$-adic analytic group if and only if $G$ is virtually a uniform pro-$p$ group, \cite[Theorem~8.18]{DDMS99}.

We start by studying uniform pro-$p$ groups using the Lazard correspondence \cite[Thm.~9.10]{DDMS99}. The connectedness of the {\color{red}proper} virtually generating graph $\Dv(U)$ of a uniform pro-$p$ group $U$ `translates' to the connectedness of the generating graph of the corresponding Lie algebra. %, which reduces to consider uniform subgroups with associated 2-generated $\mathbb{Q}_p$-Lie algebra. 
For such groups $U$, we prove that the diameter of $\Dv(U)$ is at most 2.% and the only isolated vertex of $\Gv(U)$ is the identity. 

Then, we prove that the graph $\Dv(G)$ is connected and of diameter at most $4$ for a $p$-adic analytic group $G$ with 2-generated $\mathbb{Q}_p$-Lie algebra, set of torsion elements of measure-zero {\color{red}and admitting a uniform subgroup $U$ such that $\Dv(U)=U\smallsetminus\{1\}$; see \cref{prop:improvement}. 

In \cref{sec:isolated_simple}, we characterise isolated vertices of $\Gv$ for certain compact $p$-adic simple Lie groups.  Finally, we prove \cref{propABC:virt_ab} parts (1) and (2) in \cref{Sec:maxClass} and \cref{ex:Q8}, respectively. %that the proper virtually generating graph for the pro-$p$ group of maximal class is connected and of diameter 2.
}

\subsection{Uniform pro-$p$ groups}
\label{sec:uniform-pro-p}

%We first deal with uniform pro-$p$ groups.
There is an exact correspondence between uniform pro\nobreakdash-$p$ groups and powerful Lie lattices over $\mathbb{Z}_p$, which assigns to a uniform group $U$ the $\mathbb{Z}_p$-powerful Lie lattice $\log(U)$, \cite[Theorem 9.10]{DDMS99}. 
%However, being powerful is not coherently inherited by subgroups so this correspondence is not enough to study open subgroups of a uniform pro-$p$ group. 

\begin{lemma}\label{lem:corresp}
    Let $U$ be a uniform pro-$p$ group and $H$ be a closed subgroup of $U$. Then, $H$ is open in $U$ if and only if $\log(H)$ generates the $\mathbb{Q}_p$-Lie algebra of $U$. 
\end{lemma}
\begin{proof}
If $H$ is a uniform subgroup of $U$, the result is automatic by the Lie-correspondence and so is the right-to-left implication.

Suppose now that $H$ is an open subgroup. Then, there is a uniform open subgroup $\tilde{H}\le H$. Moreover, there is $m \in \mathbb{N}$ such that $H^{p^m} \le \tilde{H}$. Since $p$-powering in $G$ corresponds to multiplication by $p$ in $L$ (see \cite[Lemma~4.14]{DDMS99}), the $\mathbb{Q}_p$-Lie subalgebras generated by $\log(H)$ and $\log(\tilde{H})$ are the same.
\end{proof}

%{\color{red}
%\begin{lemma}
%    Let $G$ be a $p$-adic analytic group with uniform open subgroup $U$ and let $H$ be a closed subgroup of $G$. Then $H$ is open if and only if $\log(H\cap U)$ generates the  $\mathbb{Q}_p$-Lie algebra of $U$.
%\end{lemma}
%\begin{proof}
%  First observe that $H$ is open if and only if $H\cap U$ is. Then the statement follows from the previous lemma
%\end{proof}
%}

By~\cref{lem:corresp}, to study the connectedness of the graph $\Dv(U)$, it is sufficient to study the generating graph of the associated $\mathbb{Q}_p$-Lie algebra.

Note that, by \cite[Prop.~4.1]{LM87} and the above argument, we just need to consider uniform pro-$p$ groups with $2$-generated associated $\mathbb Q_p$-Lie algebra.

%\subsection{Generating graphs of Lie algebras}
%\label{sec:openizers_p_adic_lie_algebras}

\begin{definition}
\label{def:openizers_p_adic_lie_algebras}
    For a $\mathbb{Q}_p$-Lie algebra $L$, the \emph{generating graph} $\Gamma(L)$ is the graph which has as vertices the elements of $L$ and two elements $X,Y\in L$ are adjacent if the $\mathbb{Q}_p$-Lie subalgebra $\langle X,Y\rangle_{Lie}$ generated by $X$ and $Y$ is the whole of $L$. We denote by $\Delta(L)$ the subgraph of $\Gamma(L)$ obtained by removing isolated vertices.
\end{definition}

In analogy with the concept of openizer, we define the following subsets of $L$.
\begin{definition}
    For $X\in L$ we define the \emph{Lie co-openizer} of $X$ as
\[
   \mathfrak{O}_X^c = \{Y\in L \mid \langle X,Y\rangle_{Lie} \neq L \}.
\]
\end{definition}

\begin{remark}
 A $\mathbb{Q}_p$-Lie algebra can be seen as an abelian profinite group and, hence, it admits a Haar measure $\mu$. However, in general the $\mathbb{Q}_p$-Lie algebra $L$ is not compact and $\mu$ is not bounded on $L$. However, it still makes sense to talk about subsets with Haar measure zero and this is the reason why we defined co-Lie openizers instead of `Lie openizers'.
\end{remark}

\begin{lemma}
\label{Lem:meraviglia}
    Let $L$ be a $2$-generated $\mathbb{Q}_p$-Lie algebra of dimension $d$.
    For every  vertex $X\in \Delta(L)$, the set $\mathfrak{O}_X^c$ has Haar measure $0$. Hence, $\mathrm{diam}(\Delta(L))\le 2$.
\end{lemma}
\begin{proof}
Let $X, Y\in L$ and define the sets~$S_1=\{X, Y\}$ and $S_{i+1}=S_i\cup [S_i,S_i]$ for $i\in \mathbb{N}$. One can show that, since the dimensions of the space is~$d$, there exists some $m$ %\leq 2\uparrow (d-1)$ (here we define inductively $2\uparrow 1 =2$ and $2\uparrow (n+1) = 2^{2\uparrow n}$ for $n\ge 1$) 
such that {\color{red}the subspace generated by $S_m$ coincides with the subspace generated by $ S_{m+j}$} for all $j\in \mathbb{N}_0$. %Indeed, put
%    \[
%    \alpha_1=2 \text{ and }\alpha_{i+1}=\binom{\sum_{k=1}^i \alpha_k}{2}
%    \]
%    for $i\leq d$ and $\sum_{k=1}^d \alpha_k\leq \alpha_{d-1}^2\leq 2\uparrow (d-1)$. Moreover, we are using the inequality $n+{ \binom{n}{2}} \le n^2$ for all $n\ge 2$.
Define the matrix~$M_X(Y)$ which has as columns the coordinates of the elements of the set~$S_m$ (considering the coordinates of $X$ as constants and those of $Y$ as variables). %This is a $d\times m$ matrix. 
Note that, for any $X\in L$,  we can rewrite the set $\mathfrak{O}_X^c$ as
\[
    \mathfrak{O}_X^c=\left\{Y\in L \mid \rk M_X(Y) <d\right\}.
    \]
 It is clear that this is a closed Zariski subset of $\mathbb{Q}_p^{N}$ for an appropriate $N\in \mathbb N$, as it is defined by the vanishing of all ${\color{red}d\times d}$ minors of $M_X(Y)$.   
    %Clearly %$X\in\mathcal{V}_X$ and we need to show $\mathcal{V}_X^c\neq \empty$. Since $L$ is 2-generated, there exists $X\in L$ not isolated in $\Gamma(L)$. Hence $\mathcal{V}_X\neq \empty$ for all $X$ not isolated in $\Gamma(L)$.
%Recalling that we removed isolated vertices from $V(L)$, 
By hypothesis, for a vertex $X\in V(\Delta(L))$, %there exists $Y\in L$ such that $\langle X,Y\rangle_{Lie}=L$. This implies that the subset $\mathcal{V}_X^c$ is non-empty. 
$L\smallsetminus\mathfrak{O}_X^c \neq \emptyset$.
Therefore, $\mathfrak{O}_X^c$ is a proper Zariski closed subset.
%     (Margulis "Semisimple algebraic groups")
 %   Since $\mathcal{V}_X$ is a closed Zariski set, by 
By \cite[2.5.3]{Mar91}, we can conclude that the Haar measure of $\mathfrak{O}_X^c$ is zero. 
 
 %If $V\subset \mathbb{Q}_p^N$ then $\mu(V)=0$.
  %  $X\in L$ not isolated iff exists $Y\in \mathcal{V}_X^c$. Hence $\mathcal{V}_X\subset  \mathbb{Q}_p^N$ and we conclude with Margulis. Hence $\mu(\mathcal{V}_X^c)=1$.
 
    Now, for two vertices in $X,Y\in V(\Delta(L))$, from the previous paragraph we can conclude that $\mathfrak{O}_X^c\cup \mathfrak{O}_Y^c$ has Haar measure zero. Hence, there exists $Z\in L$ such that $\langle X, Z\rangle_{Lie}=L$ and $\langle Y, Z\rangle_{Lie}=L$ and the diameter of $\Delta(L)$ is at most 2. 
\end{proof}

Now \cref{thm:A} is an immediate consequence of the Lazard correspondence \cite[Theorem~9.10]{DDMS99}.

%\begin{corollary}\label{cor:graph_uniform}
%    Let $U$ be a uniform pro-$p$ group with $2$-generated $\mathbb{Q}_p$-Lie algebra. Then $\Dv(U)%=\Gv(U)\smallsetminus\{1\}
%    $ is connected and of diameter at most %$2$.
%\end{corollary}

\subsection{Compact \texorpdfstring{$p$}{p}-adic analytic groups with measure-zero torsion}\label{sec:improvement}

Here we prove a small improvement of \cref{thm:A}. %We restate it for convenience.

\begin{proposition}\label{prop:improvement}
Let $G$ be a $p$-adic analytic group with lower rank at most $2$ and Haar measure of torsion elements equal to zero. Suppose additionally that, for a uniform subgroup $U$ of $G$, we have that $\Dv(U)=U\smallsetminus\{1\}$. Then, $\diam(\Dv(G))\leq4$.
\end{proposition}
\begin{proof}[Proof of \cref{thm:A}]
    Let $U$ be the uniform open subgroup of $G$ and set $\lvert G:U \rvert=n$.
%{\color{red}Is it possible that: $g\in G$ has infinite order, $g$ is not isolated in $G$, but $g^{n!}$ is isolated in $U$? Probably in this situation $g$ is only adjacent to torsion elements?}
    If $g,h\in G$ have infinite order, then $g^{n!},h^{n!}\in U$ and, by {\color{red}hypothesis}, there exists $u\in U$ that is adjacent to both $g^{n!}$ and $h^{n!}$. It is now clear that $u$ is also adjacent to both $g$ and $h$.

    Let now $g\in G$ be a non-isolated torsion element. We claim that $g$ is adjacent to an element of infinite order. Indeed, if the claim holds, by the previous paragraph, $g$ is at distance at most 3 from any other element of infinite order and of distance at most 4 to any other non-isolated torsion element.

    %Suppose now that the element $g\in G$ has finite order and it is not isolated. If $g$ is connected to some element of infinite order, by the previous paragraph it is at distance at most 3 from any other element of infinite order. On the other hand, by a similar argument, two non-isolated vertices of finite order are at distance at most 4.

    To prove the claim, %To finish the proof, 
    suppose by contradiction that $g\in G$ has finite order and it is only adjacent %connected 
    to torsion elements; take such an $h\in G$ of finite order with $A=\langle g,h\rangle \le_o G$. Under our assumptions, the set 
    \[
     S_g= \{y\in G \mid \langle g,y \rangle=A\}
    \]
     contains only torsion elements. By hypothesis, $\mu(S_g)=0$. We will now reach a contradiction by showing that the measure of $S_g$ must be positive as well.
    
    Working in the Frattini quotient $A/\Phi(A) \cong \mathbb{F}_p^2$, the set $S_g$ contains all elements of $A$ that are not multiples of $g\Phi(A)$ modulo $\Phi(A)$; that is, the measure of $S_g$ in $A$ is at least $1-\frac{1}{p}$. Hence, the measure of $S_g$ in $G$ is at least $\left(1-\frac{1}{p}\right)\mu(A)>0$, a contradiction. 
\end{proof}

{\color{red}

%We will now provide some examples that satisfy the hypotheses of \cref{prop:improvement}. We need the following recent result.

%\begin{lemma}[{\cite[Prop.~3.2.11]{Tot22}}]\label{lem:tors_meas_zero}
%    The measure of the set of torsion elements in a non-virtually-solvable $p$-adic analytic group is zero. 
%\end{lemma}

%{\color{teal}In particular, combining \cref{lem:tors_meas_zero} and \cref{thm:A}, it follows that for any non-virtually solvable $p$-adic analytic group $G$ of lower rank at most $2$ we have that $\diam(\Dv(G))\le 4$.}

%\begin{remark}\label{rmk:hypotheses}
%Let $G$ be a compact subgroup of $\mathcal{G}(\mathbb{Q}_p)$\commargherita{Do we want $\mathcal{G}(K)$ where $K$ local non-Archimedean of char $p$?}, where $\mathcal{G}$ is a simple algebraic group, and let $U$ be a uniform subgroup of $G$. By \cite[\S~1.2.2]{Boi09} and \cite[Theorem~B]{Boi09}, the Lazard $\mathbb{Q}_p$-Lie algebra $L$ of $U$ satisfies $\Delta(L)=L\smallsetminus\{0\}$. In particular, combining \cref{lem:tors_meas_zero} and this observation, it follows that for any $G$ as above we have that $\diam(\Dv(G))\le 4$.
%\end{remark}\commargherita{is this then just worse than \ref{cor:barnea_larsen} }\commatteo{we have some solvable examples below}

We will now provide some examples of solvable $p$-adic analytic groups that satisfy the hypotheses of \cref{prop:improvement}.

\begin{proposition}\label{lem:metabelian_examples}
Let $K|\mathbb{Q}_p$ be a finite extension of degree $d$, and choose an integral primitive element $\alpha\in\mathcal O_K$, so that $K=\mathbb Q_p(\alpha)$. Define the $\mathbb{Q}_p$-Lie algebra $L_{K,\alpha}= K \rtimes \langle X\rangle$ with commutator relations $[v,w]=0$  and $[v,X]=\alpha v$ for all $v,w\in K$.
Then $\Delta(L_{K,\alpha})=L_{K,\alpha}\smallsetminus\{0\}$. In particular, the uniform pro-$p$ groups $U_{K,\alpha}$ associated to the $\mathbb{Z}_p$-lattices $\Lambda_{K,\alpha}=\mathcal{O}_K\rtimes \langle pX\rangle$ via the Lazard correspondence satisfy $\Dv(U_{K,\alpha})=U\smallsetminus\{1\}$.
\end{proposition}
\begin{proof}
    Note that $L=L_{K,\alpha}$ is generated by $X$ and any nonzero $v\in K$, so it is $2$-generated. Moreover, it is of dimension $d+1$. 

    Pick a non-zero $z=v+ aX\in L$ for $a\in \mathbb{Q}_p$ and $v\in K$. If $a=0$, then $\langle v,X\rangle =L$. If $a\neq 0$, we can choose a $0\neq w\in K$ and calculate the commutator $[w,z,\ldots,z] = a^j \alpha^j w$, where $z$ appears $j$ times. It follows that $\langle z,w \rangle = K$. Hence, we have that $v\in \langle z,w \rangle$, so $X=a^{-1}(z-v)\in \langle z,w \rangle$ and it follows that $z$ is not isolated.
\end{proof}

We now exhibit some non-uniform $p$-adic analytic pro-$p$ groups to which we can apply \cref{prop:improvement}. Consider a field extension $K\vert \mathbb{Q}_p$ such that $\zeta_p\in K$ and consider the automorphism $\varphi$ of $L_{K,\alpha}$ defined by $X\mapsto X$ and $v\mapsto \zeta_p v$ for all $v\in K$. Then the group $G=U_{K,\alpha}\rtimes \langle \varphi\rangle$ is a metabelian $p$-adic analytic group with non-trivial torsion.

We will now show that the set of torsion elements in $G$ has Haar-measure zero. Setting $A=\exp(\mathcal{O}_K)$, we note that $G\cong A\rtimes (\mathbb{Z}_p \times C_p)$. Let $g=u\varphi^j \in G$ be a torsion element with $u\in U_{K,\alpha}$ and $j\in \{0,\ldots,p-1\}$. Then, projecting modulo $A$, we see that $u\in A$; thus, the set of torsion elements is contained in $A\langle \varphi\rangle$ and it has Haar-measure zero. 

One could additionally show that all torsion elements have order $p$ and that the set of torsion elements in $G$ is exactly the union of the cosets of $A\varphi^i$ for $i=1,\ldots,p-1$ and the identity.

}

\begin{remark}
    Note that there are $p$-adic analytic groups with set of torsion elements of positive measure. %: for instance, the group $\mathbb{Z}_p\rtimes C_2$ with $C_2$ acting by multiplication by $-1$. 
    By \cite{KOTWV26}, these groups are virtually-solvable.     
In the next example we show that we cannot apply the strategy of the proof of \cref{prop:improvement} to {\color{red}all} virtually solvable groups. In fact, in a virtually solvable group, it is possible that torsion elements are adjacent to only torsion elements in the virtually generating graph.
\end{remark}

\comm{
\begin{example}
 Let $V$ be the free $\mathbb{Z}_p$-module of rank $4$ with basis  $\{\mathds{1}, i, j, k\}$ and consider the representation of the quaternion group $Q_8$ on $V$ given by: for any $v\in V$
 \[
    \rho(-1).v = -v \qquad \text{ and } \qquad g.v = gv \text{\quad for } g\in Q_8\smallsetminus\{\pm1\}. 
 \]
\commatteo{I am not sure how to write this properly, what do you think about this?}
 Define the group $G=V \rtimes Q_8$ with $Q_8$ acting via the above action. Then, the element $(0,i)$ is torsion and it is only connected to other torsion elements. First or all, $(0,i)$ is not isolated, as it is connected to $(\mathds{1},j)$. In fact, $(\mathds{1},j)^2=(\mathds{1}+j,-1)$ and $(\mathds{1},j)^4=(0,1)$. On the other hand, $(0,i)^{-2} (\mathds{1},j)^{2} = (0,-1)(\mathds{1}+j,-1)=(\mathds{1}+j,1)$. It can be checked directly that $\{\mathds{1}+j,i(\mathds{1}+j),j(\mathds{1}+j),ij(\mathds{1}+j)\}=\{\mathds{1}+j,i+k,-\mathds{1}+j,i-k\}$ are linearly independent, it follows that $\langle(0,i),(\mathds{1},j)\rangle$ is open in $G$.
 
 To conclude, observe that any element of the form $(v,g)$ for $1\neq g\in Q_8$ is of finite order; in fact, $(v,g)^4=((1+g+g^2+g^3)v,1)=(0,1)$. So elements of infinite order in $G$ are of the form $(v,1)$. It is now clear that $\langle (0,i),(v,1)\rangle \cap V$ has dimension at most 2 and hence it is not open. 
% I think that this example has connected graph of diameter 3: isolated vertices are $(v,\pm 1)$ and the rest are connected in the triangle $i \to j \to k \to i$
\end{example}
}

\begin{example}\label{ex:Q8}
    Let $V$ be the free $\mathbb{Z}_p$-module of rank $4$ with basis
    $\{\mathds{1}, i, j, k\}$. The quaternion group $Q_8$ acts by left
    multiplication on the set $\{\pm 1, \pm i, \pm j, \pm k\},$ and this
    induces an action of $Q_8$ on $V$.
    Let $G=V \rtimes Q_8$ be the corresponding semidirect product. Write any
    element of $G$ in the form $(v,x)$ with $v\in V$ and $x\in Q_8$. Note
    that if $x\neq 1,$ then $(v,x)^4=((1+x+x^2+x^3)v,1)=(0,1)$ so if
    $g\notin V,$ then $|g|\leq 4.$
    Moreover if $x\neq 1,$ then $C_V(x)=\{0\}$ and this implies that, for
    every $v\in V$, $(v,x)$ and $(0,x)$ are conjugate in $G.$
    Consider the sets $\Omega_i, \Omega_j, \Omega_k$ of the elements
    $(v,x)\in G$ with, respectively, $x=\pm i, x=\pm j, x=\pm k.$ We want to
    determine $\mathbb{O}_g$ when $g\in \Omega_i.$ First assume $g=(0,i).$
    Clearly, for every $v \in V,$ $\langle (0,i),(v,1)\rangle \cap V$ has
    dimension at most $2$ and hence it is not open. Since
    $\langle (v,-1),(0,i)\rangle=\langle (v,-1)(0,i)^2,(0,i)\rangle
    =\langle (v,1), (0,i)\rangle$, it follows also that
    $V(0,-1)\cap \mathbb{O}_g=\emptyset.$ Now consider $y=(v,j).$ Then
    $$(0,i)^2(v,j)^2=(0,-1)((\mathds{1}-j)v,-1)=((j-\mathds{1})v,0)
    \in \langle g, y\rangle.$$ 
    If $v\neq 0,$ then $w=(j-\mathds{1})v\neq 0$ and $\langle g, y\rangle$ contains $\langle (w,0),(iw,0),(jw,0),(kw,0)\rangle$.     Assume in particular that $v=x_1\mathds{1}+x_2i+x_3i+x_4k.$ Then $w, iw, jw, kw$ are linearly independent if and only if 
    
$$\det\begin{pmatrix}x_1&x_2&x_3&x_4\\-x_2&x_1&-x_4&x_3\\-x_3&x_4&x_1&-x_2\\-x_4&-x_3&x_2&x_1\end{pmatrix}=(x_1^2+x_2^2+x_3^2+x_4^2)^2\neq 0.$$
    
    It follows that $\langle g,(v,j)\rangle$ is open in $G$ if $x_1^2+x_2^2+x_3^2+x_4^2\neq 0.$
    Hence, by \cite[2.5.3]{Mar91}, $\mu(\mathbb O_g\cap \Omega_j)=\mu(\Omega_j)$.
    Similarly $\mu(\mathbb O_g\cap \Omega_k)=\mu(\Omega_k).$
    Since every element of $\Omega_i$ is conjugate either to $(0,i)$ or to
    $(0,-i)$, we deduce that for every $z\in \Omega_i$, $\mathbb{O}_z$ is
    the complement in $\Omega_j\cup \Omega_k$ of a subset of zero measure.
    Clearly we have similar statements for the elements of $\Omega_j$ and
    $\Omega_k$.
    It is clear that there is no edge in $\Gv(G)$ joining two vertices
    in $V$. Finally notice that, for every $v\in V$,
    $(v,-1)=(w,i)^2$ where $w\in V$ is such that $v=(\mathds{1}-j)w.$
    Since no element of $V$ is adjacent to $(w,i)$, it follows that no
    element of $V$ is adjacent to $(v,-1).$ Putting everything together we
    reach the following conclusion:
    \begin{enumerate}
        \item the non-isolated vertices of $\Gv(G)$ are the elements of
        $\Omega_i \cup \Omega_j \cup \Omega_k$;
        \item for every non-isolated vertex $g$ of $\Gv(G)$,
        $\mu(\mathbb{O}_g)=1/2$ and $\mathbb{O}_g$ is not closed in $G$;
        \item the graph $\Dv(G)$ is connected with diameter $2$;
        \item torsion elements in $G$ are adjacent to only torsion elements in $\Dv(G)$.
    \end{enumerate}
\end{example}

\subsection{Isolated vertices in simple algebraic groups over \texorpdfstring{$\mathbb Q_p$}{Qp}}\label{sec:isolated_simple}

In light of \cref{prop:A} and \cref{thm:A}, we might wonder what are the isolated vertices of the graph $\Gv(G)$ for a compact $p$-adic analytic group $G$. In this section we show that we can exactly determine the isolated points in some cases. We need some notation from \cite{Wei96}.

Let $X$ be a simple simply-connected affine group scheme and let $K$ be a local field of characteristic zero, i.e.\ a finite extension of $\mathbb{Q}_p$, with ring of integers $\mathcal{O} \subset K$. We denote by $\mathfrak{p}$ the unique maximal ideal of $R$ and by $F=\mathcal{O}/\mathfrak{p}$ its residue field. It is well-known that the projection $\mathcal{O}\to F$ induces a short exact sequence
\begin{equation}\label{eq:ses}
   \mathcal{M} \hookrightarrow X(\mathcal{O}) \xtwoheadrightarrow[]{\pi} X(F). 
\end{equation}
Recall that a homomorphism $f:G\to H$ is said to be a Frattini $p$-cover if $\ker f \le \Phi(G)$.

\begin{theorem}[{\cite[Thm.~B]{Wei96}}]\label{thm:thomas}
If $K\vert\mathbb{Q}_p$ is unramified and $p\ge 5$, then the sequence \eqref{eq:ses} is a Frattini $p$-cover.  
\end{theorem}

 It should be noted that the theorem as stated in \cite[Thm.~B]{Wei96} holds also in some additional cases for $p=2$ and $p=3$ and our argument below also holds for these.

%By \cite[Proposition 22.13.2]{FJ23}  if $p\geq 5$ then the map $\pi_1: \mathrm{SL}_2(\Zp) \to \mathrm{SL}_2(\mathbb F_p)$ is a Frattini $p$-cover.

\begin{theorem}\label{thm:isolated}
    Let $X$ be a simple simply-connected affine group scheme, $p\ge 5$ and $K\vert \mathbb{Q}_p$ be an unramified extension. Write $\mathcal{O} \subset K$ for the ring of integers of $K$ and $F$ for its residue field. Then, $x\in X(\mathcal{O})$ is isolated in $\Gv(X(\mathcal{O}))$ if and only if $x$ has finite order and $\pi(x)\in Z(X(F))\smallsetminus\{1\}$.
\end{theorem}
\begin{proof}
{\color{red}By \cite[\S~1.2.2]{Boi09} and \cite[Theorem~B]{Boi09}, the Lazard $\mathbb{Q}_p$-Lie algebra $L$ of $\mathcal{M}$ satisfies $\Delta(L)=L\smallsetminus\{0\}$. Hence,} elements in $\mathcal{M}$ are not isolated. Moreover, by the same argument as in the proof of \cref{prop:improvement}, also elements of infinite order in $X(\mathcal{O})$ are not isolated. We will now prove that elements that project to non-central elements in $X(F)$ via $\pi$ are not isolated. 

First, we note that $X(F)$ is quasi-simple, i.e.\ perfect and simple modulo its center. We now show that its center $Z=Z(X(F))$ is contained in the Frattini subgroup $\Phi(X(F))$ of $X(F)$. In fact, if this were not the case, we would have $X(F)=Z\cdot M$ for a maximal subgroup $M$. Hence $X(F)'= M'$ and so $M=X(F)$ which would be a contradiction.

By \cite[Corollary]{GK00} and the argument from the previous paragraph, for any $\overline{x}\notin Z(X(F))$ there exists $\overline{y}\in X(F)$ such that $\langle \overline{x},\overline{y} \rangle = X(F)$. 

By \cref{thm:thomas}, the map $\pi: X(\mathcal{O}) \to X(F)$ is a Frattini $p$-cover.
So if $\pi(x)\notin Z(X(F))$, then there exists $y\in X(\mathcal{O})$ such that %$\langle \pi(x), \pi(y)\rangle = X(F),$ and consequently 
$\langle x, y\rangle=X(\mathcal{O})$. 

We will now show that, if $x$ has finite order and $\pi(x)\in Z(X(F))\smallsetminus\{1\}$, then $x$ is isolated in $\Gv(X(\mathcal{O}))$. Suppose by contradiction that this is not the case. By \cite[pg.~517]{Mil17}, the projection map restricted to the center $\pi_{\vert Z(X(\mathcal{O}))}:Z(X(\mathcal{O})) \to Z(X(F))$ is surjective. Hence, there exists a central element $z \in Z(X(\mathcal{O}))$ such that $\pi(x) = \pi(z)$. This implies that $\pi(x z^{-1}) = 1$, meaning the element $x z^{-1}$ lies in the congruence subgroup $\mathcal{M}$. Since $x$ has finite order and $z$ is central (and also of finite order, again by  \cite[pg.~517]{Mil17}), also $x z^{-1}$ is of finite order. But $\mathcal{M}$ is torsion-free, so $x = z$, which is a contradiction. 
\end{proof}

\subsection{The graph \texorpdfstring{$\Dv$}{Dvirt} for the pro-\texorpdfstring{$p$}{p} group of maximal class}
\label{Sec:maxClass}
In this section we study the connectivity of the proper virtually generating graph of the pro-$p$ group of maximal class. Note that this group is virtually isomorphic to $\mathbb{Z}_p^{p-1}$, hence, it has lower rank equal to $p-1$. We will show that its proper virtually generating graph  is connected and of diameter 2. We will write $P_i(X)=1+X+\ldots+X^{i-1}$ for $i\ge 0$. We now recall the definition of the pro-$p$ group of maximal class.

Consider a primitive $p$-th root of unity $\zeta$ over $\mathbb{Q}_p$ and form the group $G=\Zp[\zeta] \rtimes \langle s\rangle$ with $s$ acting by multiplication by $\zeta$; i.e.\ $v\mapsto \zeta \cdot v$ for any $v\in \Zp[\zeta]$. Note that $P_p(\zeta)=0$%$1+\zeta + \ldots \zeta^{p-1}=0$
, so $\mathbb{Z}_p[\zeta]$ is a free $\mathbb{Z}_p$-module of dimension $p-1$ and $s$ acts on it via the companion matrix $A$ of the polynomial $P_p(X)$. 
%\[
%  A=\begin{pmatrix}
%      0 & 0 & 0 & \ldots & 0 &-1 \\
%      1 & 0 & 0 & \ldots & 0 &-1 \\
%      0 & 1 & 0 & \ldots & 0 &-1 \\
%      \vdots &  &  & \ddots & & \vdots \\
%      0 & 0 & 0 & \ldots & 1 &-1 \\
%  \end{pmatrix}.
%\]
Note that $V=\mathbb{Q}_p\otimes_{\Zp} \Zp[\zeta]$ is a simple $\mathbb{Q}_p[C_p]$-module with respect to the matrix $A$. So any non-zero $v\in V$ generates the whole of $V$. It follows that any $0\neq v\in \Zp[\zeta]$ generates an open $\mathbb{Z}_p[C_p]$-submodule.%; that is, the subgroup $\langle v,s\rangle$ is open. 

\comm{
Note that any element $u s \in G$ with $u\neq 0$ has order $p$: 
\[
   (us)^p= (u+ \zeta\cdot u + \ldots + \zeta^{p-1} \cdot u) s^p = [(I+A + \ldots+ A^{p-1})u] 1_{C_p} = 0_{\mathbb{Z}_p[\zeta]} \ 1_{C_p} = 1_G
\]
where we used the fact that $A$ satisfies its minimal polynomial by Cayley-Hamilton.
}
We work in the cosets $\Zp[\zeta]s^i$ of the uniform open subgroup $\Zp[\zeta]$, $i=0,\ldots,p-1$. 
First observe that no two elements $u,v\in \Zp[\zeta]$ are connected if $p\ge 3$, because $\mathbb{Z}_p[\zeta]$ is of rank $p-1$. 

We will now show that all elements $v_1 s^i$ and $v_2 s^j$ are connected when $(i,j)\neq (0,0)$ and one is not a power of the other. %Without loss of generality we can suppose that $i\neq 0$. 
Note that $P_i(X)\cdot(X-1)=X^i-1$. We compute
\[
   (v_1 s^i)^j = [(1+\zeta^i + \zeta^{2i} \ldots + \zeta^{(j-1)i})v_1] s^{ij} = [P_j(\zeta^i) v_1 ] s^{ij}
   %\left[ \frac{\zeta^{ij}-1}{\zeta^i-1} v_1 \right] s^{ij} 
\]
and similarly $(v_2 s^j)^i = [ P_i(\zeta^j) v_2] s^{ij}
%\left[\frac{\zeta^{ij}-1}{\zeta^j-1} v_2 \right] s^{ij} 
$. If we multiply $(v_1 s^i)^j(v_2 s^j)^{-i}$ the top components cancel out and we are left with $P_j(\zeta^i) v_1 - P_i(\zeta^j) v_2
    %\frac{\zeta^{ij}-1}{\zeta^i-1} v_1 - \frac{\zeta^{ij}-1}{\zeta^j-1} v_2 = (\zeta^{ij}-1) \left( \frac{v_1}{\zeta^i-1} - \frac{v_2}{\zeta^j-1}  \right) 
$. %Multiplying by $(\zeta^i-1)(\zeta^j-1)$ we obtain
%\[
%    (\zeta^{ij}-1)(\zeta^i-1)v_1 - (\zeta^{ij}-1)(\zeta^j-1)v_2 
%\]
Assume without loss of generality that $i\neq 0$. Since %any non-zero element in $\mathbb{Z}_p[\zeta]$ generates an open submodule and conjugation by $v s^k$ generates the same subspace as conjugation by $s$, 
$\langle v_1 s^i,v_2 s^j \rangle=\langle  (v_1 s^i)^j(v_2 s^j)^{-i}, v_1 s^i\rangle$, if $w=(v_1 s^i)^j(v_2 s^j)^{-i}\neq 0$, then $(v_1 s^i) w (v_1 s^i)^{-1}= \zeta^i \cdot w$ and it follows that this subgroup is open. Hence, we obtain that $v_1 s^i$ and $v_2 s^j$ generate an open subgroup whenever $%\label{eq:non-open}
   P_j(\zeta^i) v_1 \neq P_i(\zeta^j) v_2$.

We now show that this equation is not satisfied exactly when $v_1 s^i$ and $v_2 s^j$ are one a power of the other; at the same time this will imply that $v_1 s^i$ and $v_2 s^j$ do not generate an open subgroup in this case. Without loss of generality, suppose that $i\neq 0$, so $P_i(\zeta)\neq 0$. Moreover, there is $k\in \{1,\ldots,p-1\}$ such that $j=ik$. Thus, using the equality $P_{ik}(\zeta) = P_i(\zeta) \cdot P_k(\zeta^i)$, we have that
\[
  (v_1 s^i)^k = \left[ %\frac{\zeta^{ik}-1}{\zeta^i-1} 
  P_k(\zeta^i) v_1 \right] s^{ik} = \left[ %\frac{\zeta^{j}-1}{\zeta^i-1} 
  \frac{ P_j(\zeta)}{P_i(\zeta)}  v_1 \right] s^{j} = v_2 s^j.
\]
It follows that all vertices in $\Zp[\zeta]s^i$ for $i\neq 0$ are connected to all other vertices, while the vertices $x\in \Zp[\zeta]$ are only connected to vertices of the other cosets. %Note that $\mu(\mathbb{O}_x) = \frac{p-1}{p}<1$ for $x\in \Zp[\zeta]$.

%%%%%%%%%%%%%%%%%%%%%%%%%%%%%%%%%%%%%%%%%%%%%%%%%%%%%%%%%%

\section{Lie algebra methods}
\label{sec:Lie-algebra-methods}

Let $G$ be a pro-$p$ group with a filtration of open normal subgroups $\{G_n\}_{n\in\mathbb{N}}$ such that $[G_n,G_m]\le G_{n+m}$ and $G_n^p \le G_{pn}$ for all $n,m$. We can define a graded Lie algebra in the following way:
\[
   L(G)= \bigoplus_{n\in \mathbb{N}} G_n/G_{n+1}.
\]
We define the sum in $L(G)$ coordinate-wise. Moreover, we can define a Lie bracket on $L(G)$ using the commutator in $G$: for $xG_{n+1}\in G_n/G_{n+1}$ and $yG_{m+1} \in G_m/G_{m+1}$
\[
   [xG_{n+1},yG_{m+1}] = [x,y] G_{n+m+1}.
\]
Furthermore, there is a natural projection $\pi:G \to L(G)$ given by $x\mapsto xG_{n+1}$ for the unique $n$ such that $x\in G_n\smallsetminus G_{n+1}$.
For a subgroup $H\le G$, we define the graded Lie algebra associated to $H$ as 
\[
  L(H)= \bigoplus_{n\in \mathbb{N}} (H\cap G_n)G_{n+1}/G_{n+1}.
\]

It is well known that the graded Lie algebra of the group $G=\mathrm{SL}_2^1(\mathbb{F}_p[\![t]\!])$ is isomorphic to $\mathfrak{sl}_2(\mathbb{F}_p) \otimes_{\mathbb{F}_p} t\mathbb{F}_p[t] $ with the natural grading.

The next lemma is straightforward.

\begin{lemma}
\label{lem:grad-Lie-alg-open}
    Let $G$ be a pro-$p$ group and consider its graded Lie algebra $L=L(G)$. If the images of two elements $x,y\in G$ generate a graded sub-Lie algebra of finite co-dimension in $L$, then $\langle x,y\rangle \le_o G$.
\end{lemma}
\begin{proof}
Set $H=\langle x,y\rangle$. If $\pi(H)=\langle \pi(x),\pi(y)\rangle_{Lie}$ is of finite co-dimension in $L$, then there is $n\in \mathbb{N}$ such that $\pi(H)$ contains $L_n=\bigoplus_{m\ge n} G_m/G_{m+1}$. It follows that $L_n \le \pi(H) \le L(H)$ and, thus, $(H\cap G_m) G_{m+1}=G_{m+1}$ for all $m\ge n$. In particular, we have that $H\ge G_n$ and $H$ is open in $G$. 
\end{proof}

In the rest of this section we denote by $L$ the graded Lie algebra associated to $G$ with respect to a fixed filtration and by $\pi:G\to L$ the natural projection. For a fixed $x\in G$, we consider the following set
\[
   \mathbb{LO}_x = \{y\in G \mid \langle \pi(x),\pi(y)\rangle_{Lie} \text{ has finite co-dimension in } L \},
\]
which can be thought as the \emph{Lie openizer} in the graded Lie algebra associated to $G$. 

\begin{remark}
    In \cref{sec:uniform-pro-p} we considered openizers in Lie algebras associated to $p$-adic analytic groups via the Lazard correspondence. We remark that in this section we consider openizers in \emph{graded Lie algebras} that can be associated to any pro-$p$ group. Of course, usually graded Lie algebras retain much less information. See for instance \cref{rmk:nottingham}.  
\end{remark}

\begin{lemma}\label{lem:prob_sl2}
 Let $0\neq X\in \mathfrak{sl}_2(\mathbb{F}_p)$. The probability that a random element $Y\in \mathfrak{sl}_2(\mathbb{F}_p)$ together with $X$ generates $\mathfrak{sl}_2(\mathbb{F}_p)$ is at least $\frac{(p-1)^2}{p^2}$.
\end{lemma}
\begin{proof}
    Consider the characteristic polynomial of $X$, this is $\lambda^2 + \det(X)=0$. If $X$ is nilpotent, then $\det(X)=0$ and $X$ belongs to exactly one Borel subalgebra 
    %\commatteo{since it is nilpotent, its kernel has dimension exactly 1, hence it fixes exactly one line and so it is in the associated Borel. The point is that Borel <-> eigenvalues} 
    of $\mathfrak{sl}_2(\mathbb{F}_p)$; the elements that together with this $X$ do not generate are exactly the elements of this Borel subalgebra. Hence, the probability that another random element together with $X$ generates is $1-\frac{p^2}{p^3} = \frac{p-1}{p}$.

    If $\det(X)$ is a non-zero quadratic residue modulo $p$, then $X$ belongs to exactly two Borel subalgebras. By the inclusion-exclusion principle, the probability that another random element together with $X$ generates is $1-\frac{2p^2-p}{p^3} = \frac{(p-1)^2}{p^2}$.

    Finally, if $\det(X)$ is not a quadratic residue modulo $p$, then any element outside the subspace generated by $X$ generates together with $X$ the whole algebra.
    %\commatteo{it is not contained in any Borel}
    Hence, the probability that another random element together with $X$ generates is $1-\frac{p}{p^3} = \frac{p^2-1}{p^2}$.

    The claim follows after comparing the values obtained.
\end{proof}

 \comm{
    
    Let's look at generation in $\mathfrak{sl}_2(\mathbb{F}_p)$. A couple $X,Y$ will fail to generate the algebra if an only if they both Lie in a Borel subalgebra. There are $p+1$ Borel subalgebras. Each of them has $p^2$ elements and $\lvert B_1\cap B_2\rvert=p$.

    Fix $X$. If $X$ is not in any Borel subalgebra then any other linearly independent element generates the algebra and the probability is $1-1/p$.
    
    If $X$ is in exactly one Borel $B$, then any element $Y$ outside of $B$ will generate with $X$, so the probability of picking such a $Y$ is $(p^3-p^2)/p^3 > \ast$.

We are left with the case when $X\in B_1\cap B_2$. Then, any element $Y$ outside $B_1\cup B_2$ will generate $\mathfrak{sl}_2(\mathbb{F}_p)$ together with $X$. Since $\lvert B_1\rvert = \lvert B_2\rvert = p^2$, the number of such elements is $p^3-2p^2+p$ (inclusion-exclusion) and the probability of picking such an element is
\[
   \frac{p^3-2p^2+p}{p^3} = \frac{p(p-1)^2}{p^3} = (1-\frac{1}{p})^2. \qedhere
\]
}

\begin{proposition}\label{prop:Lie_SL2}
Let $p\ge 5$. For all $1\neq x\in G=\mathrm{SL}_2^1(\mathbb{F}_p[\![t]\!])$, we have that $\mu(\mathbb{LO}_x)>1/2$. In particular, the graph $\Dv(G)$ is connected and of diameter at most $2$.
\end{proposition}
\begin{proof}
 Write $L=\bigoplus_n \mathfrak{sl}_2(\mathbb{F}_p)\otimes t^n$ for the associated graded Lie algebra and $\pi:G\to L$ for the natural projection. Take $1\neq x\in G$ and write $\pi(x)=X\otimes t^n$, for $X\in \mathfrak{sl}_2(\mathbb{F}_p)$ with $X\neq 0$. Consider another element $y\in G$ and its projection $\pi(y)=Y\otimes t^m$.

    We observe that $\langle \pi(x),\pi(y)\rangle_{Lie}$ is of finite co-dimension in $L$ if and only if $n$ and $m$ are coprime and  $\langle X,Y \rangle=\mathfrak{sl}_2(\mathbb{F}_p)$. Therefore, using \cref{lem:prob_sl2} we obtain
    \begin{align*}
    \mu(\mathbb{LO}_x) & =  \underset{\mathrm{gcd}(n,m)=1}{\sum_{m=1}^\infty} \mu(\{ y\in G \mid \pi(y)=Y\otimes t^m \text{ and } \langle X,Y \rangle=\mathfrak{sl}_2(\mathbb{F}_p)  \}) \\ & = \underset{\mathrm{gcd}(n,m)=1}{\sum_{m=1}^\infty} \frac{1}{p^{3(m-1)}} \cdot \mu(\{Y\in \mathfrak{sl}_2(\mathbb{F}_p) \mid \langle X,Y \rangle= \mathfrak{sl}_2(\mathbb{F}_p)\}) \\ &\ge \frac{(p-1)^2}{p^2} \underset{\mathrm{gcd}(n,m)=1}{\sum_{m=1}^\infty} \frac{1}{p^{3(m-1)}}  \ge \frac{(p-1)^2}{p^2} >\frac{1}{2}.
    \end{align*}  
   In the second equality above we have used the fact that the set of elements in $G$ with projection of degree $m$ is the subgroup $\gamma_m(G)$ and this has index $p^{3(m-1)}$ in $G$.
    
    Using \cref{lem:grad-Lie-alg-open} and \cref{lem:big_open} we can conclude that the graph $\Dv(G)$ is connected and of diameter at most $2$.
\end{proof}

The fact that $\Dv(\mathrm{SL}_2^1(\mathbb{F}_p[\![t]\!]))$ is connected and of diameter 2 was already established in {%\cref{cor:barnea_larsen}
\color{red}\cref{prop:A}} with a different method. However, we will show that the approach from this section can be used to show that also the subgroups $\mathcal{Q}^1(s,r)$ of $\mathcal{N}(\mathbb{F}_p)$ ($0<r<p^s/2$) defined in \cite{Ers04} as analogues of the group $\mathrm{SL}_2^1(\mathbb{F}_p[\![t]\!])$ have connected proper virtually generating graph of diameter two.  We will not recall the definition of the groups $\mathcal{Q}^1(s,r)$ and we will only highlight the results from \cite{Ers04} that we need. First we recall a definition from \cite{Ers04}.

\begin{definition}
    Let $G$ and $H$ be pro-$p$ groups with fixed filtrations $\{G_n\}_{n\in \mathbb N}$ and $\{H_n\}_{n\in \mathbb N}$. A map $\varphi : G \to H$ is called an \emph{approximation} if
    \begin{itemize}
        \item[(a)] $\varphi$ is bijective and $\varphi(G_n) = H_n$ for all $n$;
        \item [(b)] there exists a positive integer $k$ such that for any $x \in G_m$ and $y \in G_n$ we have $\varphi(xy) \equiv \varphi(x)\varphi(y) \mod H_{m+n+k}$.
    \end{itemize}
\end{definition}

By \cite[Prop.~4.3]{Ers04}, there exists an approximation $\varphi_{s,r}: \mathrm{SL}_2^1(\mathbb{F}_p[\![t]\!]) \to \mathcal{Q}^1(s,r)$. In the next proposition we deduce some consequences of the existence of such a map.

\begin{lemma}\label{lem:approx}
     Let $G$ and $H$ be pro-$p$ groups with fixed filtrations $\{G_n\}$ and $\{H_n\}$ and let $\varphi : G \to H$ be an approximation. Then, for any measurable subset $X\subseteq G$, we have that $\mu_G(X)=\mu_H(\varphi(X))$.
\end{lemma}
\begin{proof}
    It is sufficient to show that $\varphi(xG_n)=\varphi(x)H_n$ for all $x\in G=G_0$ and all $n$.

    By hypothesis, for $y\in G_n$, $\varphi(xy)=\varphi(x)\varphi(y)h$ for some $h\in H_{n+k}$. Hence, $\varphi(xG_n) \subseteq \varphi(x)H_n$. Now, writing $G=\bigsqcup_{i=1}^d x_i G_n$ and applying $\varphi$ to this equality we obtain
    \[
      H=\varphi(G) = \bigcup_{i=1}^d \varphi(x_i G_n)\subseteq \bigcup_{i=1}^d \varphi(x_i) H_n
    \]
    and it follows that $ \lvert H:H_n\rvert \le \lvert G:G_n\rvert$. We will now prove that different cosets modulo $G_n$ map to different cosets modulo $H_n$ via $\varphi$ and this will conclude the proof.

    Suppose that $x_1,x_2\in G=G_0$ and $\varphi(x_1)\equiv \varphi(x_2) \mod H_{n}$. Then, for the element $w_0=\varphi(x_1)^{-1} \varphi(x_2) \in H_n$ there exists $y_0\in G_n$ such that $\varphi(y_0)=w_0$. It follows that 
    \[
      \varphi(x_1 y_0) \equiv \varphi(x_1) \varphi(y_0) \equiv \varphi(x_2)\mod H_{n+k}
    \]
    In the same way, we can inductively construct elements $y_i \in G_{n+ik}$ such that $\varphi(x_1 y_0 \cdots y_i) \equiv \varphi(x_2) \mod H_{n+ik}$. By construction, the sequence $(y_0\cdots y_i)_{i\in \mathbb{N}}$ converges to an element $y$ in the closed subgroup $G_n$ and 
    \[
      \varphi(x_1 y) \equiv \varphi(x_2) \mod \bigcap_{i\in \mathbb N} H_{n+ik} \Longrightarrow \varphi(x_1 y) = \varphi(x_2).
    \]
    
    Since $\varphi$ is bijective, we must have that $x_1 y=x_2$; that is, $x_1 G_n = x_2 G_n$, as claimed.
    \end{proof}

\begin{corollary}
\label{cor:groupsQ1}
  Let $p\ge 5$, then the graph $\Dv(\mathcal{Q}^1(s,r))$ is connected and of diameter at most $2$.
\end{corollary}
\begin{proof}
    The map $\varphi_{s,r}: \mathrm{SL}_2^1(\mathbb{F}_p[\![t]\!]) \to \mathcal{Q}^1(s,r)$ is an approximation by \cite[Prop.~4.3]{Ers04}. By \cite[Prop.~4.2(1)]{Ers04}, the approximation $\varphi_{s,r}$ induced an isomorphism of graded Lie algebras $L(\mathrm{SL}_2^1(\mathbb{F}_p[\![t]\!]))\cong L(\mathcal{Q}^1(s,r))$. The result follows from \cref{lem:approx} and the proof of \cref{prop:Lie_SL2}.
\end{proof}

%By Theorem~1.1 in Misha's paper \cite{Ers04} \url{https://m-ershov.github.io/Research/nottf.pdf}, the graded Lie algebra of $\mathcal{Q}^1(r,s)$ is isomorphic to that of $\mathrm{SL}_2^1(\mathbb{F}_p[\![t]\!])$, so the same result holds for these groups. {\color{red}define the groups, maybe this section should go after the Nottingham group one...}

\begin{remark}\label{rmk:nottingham}
  We note here that the above argument for $\mathcal{Q}^1(s,r)$ does not work for the Nottingham group. In fact, it can be shown that elements $g\in \mathcal{N}$ such that $D(g)\equiv_p 0$ have Lie openizer $\mathbb{LO}_g$ of measure zero. We will prove that the graph $\Dv(\mathcal{N})$ is connected using a different method in the next section. 
\end{remark}

%%%%%%%%%%%%%%%%%%%%%%%%%%%%%%%%%%%%%%%%%%%%%%%%%%%%%%%%

\section{Nottingham group}\label{sec:Nottingham}

The Nottingham group over $\mathbb{F}_p$, denoted by $\mathcal{N}=\mathcal{N}(\mathbb{F}_p)$, is the group of formal power series of the form 
    \[
        f= t +\sum_{n=1}^\infty f_n t^{n+1}\in\mathbb{F}_p[\![t]\!]
    \]
under the formal substitution. Equivalently, it is the Sylow pro-$p$ subgroup of the local field $\mathbb{F}_p(\!(t)\!)$. We assume for simplicity that $p\geqslant5$. In this section we prove that the proper virtually generating graph $\Dv(\mathcal{N})$ of the Nottingham group $\mathcal{N}$ is connected of diameter $2$. Our argument builds on some ideas presented in Section~5 of~\cite{Cam00}, which, in turn, were used by Mori~\cite{Mor22} to prove that diameter is at most 4, and takes the analysis somewhat further.

First, we recall some definitions and basic properties of the Nottingham group, see~\cite{Cam00} for more details. For every $k\in\mathbb{N}$, let
    \[
        \mathcal{N}_k=\{f\in\mathcal{N}\mid f\equiv t\mod{t^{k+1}}\}.
    \] 
Then  $\{\mathcal{N}_k\}_{k\in \mathbb{N}}$ is a filtration of open normal subgroups of $\mathcal{N}$ satisfying $|\mathcal{N}:\mathcal{N}_k|=p^{k-1}$. It can be shown that $\mathcal{N}$ is the inverse limit of the quotients $\mathcal{N}/\mathcal{N}_k$.
For any non-trivial $f\in \mathcal{N}$, the minimal $n\in\mathbb{N}$ such that $f\in\mathcal{N}_n\smallsetminus\mathcal{N}_{n+1}$ is called the \emph{depth} of~$f$, and we denote it by~$D(f)$. The depth of the identity is defined to be~$\infty$. 
One can show that for each non-trivial $f\in\mathcal{N}$ with $D(f)=n$, we have $\langle f, \mathcal{N}_{n+1}\rangle=\mathcal{N}_n$, hence $\mathcal{N}_k$ is generated by any set of elements of $\mathcal{N}$ containing an element of every depth greater than $k-1$. Consequently, given $f,g\in\mathcal{N}$, the subgroup $\langle f,g\rangle$ is open if there exists $k\in\mathbb{N}$ such that, for all integers $\overline{m}\geqslant k$, there exists an element $h\in\langle f,g\rangle$ with $D(h)=\overline{m}$.

A key step in producing elements of arbitrary large depth is taking commutators. Indeed, if  $f,g\in\mathcal{N}$ with $D(f)=n$ and $D(g)=m$, one has
\begin{equation}
    \label{eq:DepthComm}
    D([f,g])=n+m\quad \text{if }n\not\equiv m\pmod{p}.
    %D([f,g])&>n+m\quad \text{if }n\equiv m\pmod{p}.   
\end{equation}

 Similarly as in the proof of \cite[Theorem 7]{Cam00}, motivated by (\ref{eq:DepthComm}), we introduce the following combinatorial notion which will allow us to control how depths combine under repeated commutators of two elements with depths $q$ and $r$.

%{\color{red}To study the virtually generating graph of $\mathcal{N}$, we will rely on the following combinatorial definition, which arose in the proof of the fact that $\mathcal{N}$ has lower rank $2$ (cf.\ \cite[Theorem 7]{Cam00}).}

\begin{definition}
Let $\overline{m}$, $q$, $r \in \mathbb N$ and $p$ a prime. A sequence of integers $\{b_1,\ldots,b_k\}$ is called a \emph{$(q, r)$-allowable partition for
$\overline{m}$} if the following conditions hold:

\smallskip

\begin{enumerate}
    \item[(A1)] $b_i$ is either $q$ or $r$, for all $i\in \{1,\dots,k\} $.
    \item[(A2)] If $s_j =\displaystyle \sum_{i=1}^j b_i$ then $s_j \not \equiv b_{j+1}\pmod{p}$, for $j\in\{1,\dots, k-1\}$.
    \item[(A3)] $\displaystyle\sum^k_{i=1} b_i = \overline{m}$.
\end{enumerate}
\end{definition}

%In order to study the virtually generating graph, we want to understand what conditions on the depths $q$ and $r$ of $f,g\in\mathcal{N}$ ensure that $\langle f, g\rangle$ is an open subgroup.
%If there exists $N\in\mathbb{N}$ such that $m$ has a $(q,r)$-allowable partition for all $m\geqslantN$ the subgroup generated ba an element of depth $q$ and one of depth $r$ is open in $\mathcal{N}$.
%{\color{red}Later, we will search for $(q,r)$-allowable partitions for the depths $q=D(g)$ and $r=D(h)$ of two elements $g,h\in \mathcal{N}$. We will show that, under certain conditions on $q$ and $r$, any $m$ large enough admits a $(q,r)$-allowable partition (cfr.\ \cref{prop:qr}). From this it will follow that elements of certain degrees always generate an open subgroup in $\mathcal{N}$ (cfr.\ \cref{rem:Closed-Open_Nott}).}
The following lemma was mentioned in \cite[\S 5]{Cam00} without proof. Its proof was given explicitly in Italian in \cite{Mor22}. We report it here for completeness. 

\begin{proposition}\label{prop:qr}
 Let $p \ge 5$ be a prime and let $q$ and $r$ be natural numbers such that
 %\begin{enumerate}
    $$ (1)\ q, r \not\equiv_p 0. \qquad (2)\ \mathrm{gcd}(q, r) = 1. \qquad (3)\ q \not\equiv_p r. \qquad (4)\ q + r \not\equiv_p 0.$$%\commatteo{number 4 is not in Rachel's article, I wounder if it is needed... It is needed! Otherwise $k$ below could be $=1$!}
 %\end{enumerate}
Then there exists an $N\in \mathbb{N}$ such that for all $\overline{m}\ge N$, $\overline{m}$ admits a $(q, r)$-allowable partition.
\end{proposition}
\begin{proof}
 We first show that there is a natural number $N$ such that for all $\overline{m}\ge N$, $\overline{m} = uq + vr$, where $u, v \in \mathbb{N}_0$, $u \ge v$ and $u - v \le q+r$.
 
Since $\mathrm{gcd}(q,r)=1$, there are $\alpha, \beta \in \mathbb Z$ such that $1 = \alpha q + \beta r$ and, without loss of generality, we can suppose that $\alpha < 0$.
Dividing $\overline{m}$ by $q$ we get $\overline{m} = xq + d$ with $0\le d < q$, $x \in \mathbb N$.
Hence:
\[
\overline{m} = xq + d(\alpha q + \beta r) = (x + d\alpha)q + (d\beta)r.
\]
Now, noting that $d<q$ for any $\overline{m}$ and $\alpha,\beta$ are fixed, if $\overline{m} \ge N=q^2(\beta -\alpha+1)$, then with the above notation $x+d\alpha \ge d\beta$. Take $u=x+d\alpha$ and $v=d\beta$.  %\commatteo{$q(\beta-q\alpha+d\alpha) \ge d\beta$}

Now, we can repeatedly replace $u$ by $u-r$ and $v$ by $v+q$ without changing the sum, but maintaining $u\ge v \ge 0$. Repeating this procedure as long as $u\ge v\ge 0$, at the last step we obtain $u-v\le q+r$.%\commatteo{the next $+q$ and $-r$ would swap them, it means that they are at most $q+r$ apart}

We are ready to construct a $(q,r)$-allowable partition for $\overline{m}$. Construct a sequence according to (A1) and (A2) above, choosing $q$ whenever possible. By hypothesis (4), there exists a minimal $k>1$ such that $kq+r\equiv_p 0$.
The choice will go as follows:
\[
  q,\ r,\ \underbrace{q,\ q,\ldots,q}_k,\ r,\ \underbrace{q,\ q,\ldots,q}_k,\ r,\ q\ldots
\]%\commatteo{in the thesis $k=2$ is possible, but then the number of $q$'s wrt $r$'s does not increase! Can we put $k$ $q$'s there? I think YES, by minimality $iq+r$ all have different congruence mod $p$, then $kq+r\equiv 0$, so the next can be $q\not\equiv 0$ and then $(k+1)q+r \equiv q$, so I have to put $r$, okay}
Thus, after a bounded number of steps, say $s$, the number of $q$'s will exceed the number of
$r$'s by $u - v$. As $u - v$ is bounded independently of $\overline{m}$, by choosing $\overline{m}$ sufficiently
large we may assume that the sum $\sum_{i=1}^s b_i$ of the terms chosen so far is less than $\overline{m}$. Now, complete the sequence so that (A3) is satisfied by taking the same appropriate number of $q$'s and $r$'s. We denote the sequence just obtained by $\underline{b}=\{b_1,\ldots,b_k\}$.

Finally, we will now prove that we can rearrange the last $k-s$ terms $b_{s+1},\ldots,b_k$ of the sequence $\underline{b}$ so that also (A2) is satisfied; we proceed by adjusting two consecutive terms each the time. By construction, the sequence $\{b_1,\ldots,b_s\}$ satisfies (A2). Suppose that $\ell\ge s$ and that $\{b_1,\ldots,b_{\ell}\}$ satisfies (A2). Denote the total sum of the sequence so far by $x=b_1+\ldots+b_{\ell}$. We choose the next two terms $b_{\ell+1}$ and $b_{\ell+2}$ as follows:
\begin{itemize}[leftmargin=15pt]
    \item If $x\not\equiv_p q,r$ and $x+r \not\equiv_p q$, then we choose $b_{\ell+1}=r$ and $b_{\ell+2}=q$.
    \item If $x\not\equiv_p q,r$ and $x+r \equiv_p q$, then by (3) we have that $x+q \not\equiv_p r$. Thus, we can continue the sequence with $b_{\ell+1}=q$ and $b_{\ell+2}=r$.
    \item If $x\equiv_p q$, then, by (1), $x+r \not\equiv_p q$ and we can continue the sequence with $b_{\ell+1}=r$ and $b_{\ell+2}=q$. 
    \item If $x\equiv_p r$, then, by (1), $x+q \not\equiv_p r$ and we can continue the sequence with $b_{\ell+1}=q$ and $b_{\ell+2}=r$.
\end{itemize}
%since if $x \equiv_p q$ or $x + q \equiv_p r$, then $x \not\equiv_p r$ and $x +r \not\equiv_p q$. 
It follows that the chosen sequence satisfies conditions (A1), (A2) and (A3).
\end{proof}

\comm{

$$ \mathcal{N}_m =\{g\in \mathcal{N} \mid g \equiv t \mod t^{m+1}\}$$

\[
   \mathcal{N}_1=\mathcal{N}
\]

\begin{lemma}
    The set $X=\{g\in \mathcal{N} \mid D(g-\mathrm{LT}(g))-D(g)\ge 3\}$ has measure $\frac{1}{p^2}$.
\end{lemma}
\begin{proof}
    The condition just means that these are the elements $g=t(1+\sum_{i=k}^\infty a_i t^{i})$ with $a_k\neq 0$ and $a_{k+1}=a_{k+2}=0$. {\color{red}Since the power series are homeomorphic to a cartesian product of $\mathbb{F}_p$'s, the Haar measure on $\mathcal{N}$ can be calculated coefficient by coefficient.} If we define $E_n=\mathcal{N}_n \smallsetminus \mathcal{N}_{n+1}$, it follows that for $n\ge 1$ 
    \[
    \mu(E_n) = \left(\frac{1}{p}\right)^{\hspace{-0.1cm}n-1} \cdot \frac{p-1}{p} 
    \]
     \[
     \mu(X\cap E_n) = \left(\frac{1}{p}\right)^{\hspace{-0.1cm}2} \left(\frac{1}{p}\right)^{\hspace{-0.1cm}n-1} \cdot \frac{p-1}{p} = \frac{p-1}{p^{n+2}}.
    \]
    Therefore, $$\mu(X) = \sum_{n=1}^{\infty} \frac{p-1}{p^{n+2}} = \frac{p-1}{p^3} \sum_{n=0}^{\infty} \frac{1}{p^{n}}=\frac{p-1}{p^3} \frac{1}{1-\frac{1}{p}} =\frac{1}{p^2}. \qedhere$$
\end{proof}

}

%{\color{red}Given a non-trivial element $g\in\mathcal{N}$, we want to compute a lower bound for the measure of the openiser. To do so, we will show that for any non-trivial $g\in \mathcal{N}$ we can find certain elements of degree one whose commutators with $g$ have degrees that are allowable partitions as in the previous Proposition. }

We collect the results of the previous discussion in the following.
\begin{remark}
    \label{rem:Closed-Open_Nott}
    Combining equation (\ref{eq:DepthComm}) and \cref{prop:qr}, every closed subgroup which contains two elements $f,g\in\mathcal{N}$ of depth $D(f)=q$ and $D(g)=r$ respectively satisfying the hypotheses of \cref{prop:qr}, contains $\mathcal{N}_N$ for some $N\in\mathbb{N}$ and so it is open.
\end{remark}

We now want to compute a lower bound for the measure of the openizers of non-trivial elements of $\mathcal{N}$. To do so, we need to compute the first few terms of some  commutators. This was done by York \cite{Yor90}, whose thesis is difficult to access; we refer the reader to the more accessible \cite[Lemma 2.1]{Klo00} and report it here, as we will use it repeatedly.

\begin{lemma}
\label{lem:coeff-Comm-Nott}
Let $f,g$ be two elements of $\mathcal{N}$, and suppose that there are $n,m,l,k\in\mathbb{N}$ such that 
    \begin{align*}
        f &\equiv t+ f_nt^{n+1}+f_{n+l}t^{n+l+1}\pmod{\mathcal{N}_{n+l+1}}\\
        g &\equiv t+ g_mt^{m+1}+g_{m+k}t^{m+k+1}\pmod{\mathcal{N}_{m+k+1}}.
    \end{align*}
Let $s:=\min \{ n+m+l,\ n+2m,\ n+m+k,\ 2n+m\}$ and
\begin{align*}
    \gamma_1 & =\begin{cases}
                    (n+l-m) f_{n+l}g_m&  \text{if }  s= n+m+l\\
                    0& \text{otherwise}
                \end{cases}
\\
        \gamma_2 & =\begin{cases}
                    \left(\binom{n+1}{2}-(m+1)(n-m)\right) f_{n}g_m^2&  \text{if } s= n+2m\\
                    0& \text{otherwise}
                \end{cases}
\\
    \gamma_3 & =\begin{cases}
                    -(m+k-n) f_{n}g_{m+k}&  \text{if } s= n+m+k\\
                    0& \text{otherwise}
                \end{cases}
\\
    \gamma_4 & =\begin{cases}
                    -\left(\binom{m+1}{2}-(n+1)(m-n)\right) f_{n}^2g_m&  \text{if } s= 2n+m\\
                    0& \text{otherwise}
                \end{cases}
\end{align*}
and put $\gamma=\gamma_1+\gamma_2+\gamma_3+\gamma_4$. Then
\[ [f,g]\equiv t+(n-m)f_ng_m t^{n+m+1}+\gamma t^{s+1}\pmod{\mathcal{N}_{s+1}}. \]
\end{lemma}

We now carry out a case-by-case analysis according to congruences classes modulo~$p$ of the depth of a given element $x\in\mathcal{N}$.

\begin{lemma}\label{lem:Nott10-1}
    Let $n\in\mathbb{N}$ and suppose that $n$ is congruent to either $1$, $0$, or $-1$ modulo~$p$. Let $x\in \mathcal{N}$ with $D(x)=n$, then $\mu(\mathbb{O}_x)\ge (p-1)^2\mu(\mathcal{N}_3)$.
\end{lemma}
\begin{proof}
There exist $l\in \mathbb{N}$, $a_n,a_{n+l}\in \mathbb{F}_p$ such that $a_n\neq 0$ and
    \begin{align*}
    x &\equiv t+a_n t^{n+1}+a_{n+l} t^{n+l+1}\pmod{ \mathcal{N}_{n+l+1}}.
    \end{align*}
Consider an element $y\in\mathcal{N}$ of depth 1, such that 
  \[
    y  \equiv t+\alpha t^2 + \beta t^3 \pmod{\mathcal{N}_{3}}
  \]
for some $(\alpha,\beta)\in\mathbb{F}_p^\ast\times\mathbb{F}_p$. We will show that, for many choices of $\alpha$ and $\beta$, the subgroup $\langle x,y\rangle$ is open.

    \smallskip

    \noindent\textit{Case 1:} $n\equiv_p 1$. By \cref{lem:coeff-Comm-Nott}, we have 
        \[
            [x,y] \equiv t +  \left( \delta_{l,1} \, a_{n+1} \alpha  +a_n \alpha^2 -a_n \beta - \delta_{n,1} a_{ n}^2 \alpha\right) t^{n+3} \pmod{\mathcal{N}_{n+3}},
        \]
    where $\delta_{i,j}$ denotes the Kronecker delta.
    Since $a_n$ is non-zero, for each $\alpha\in\mathbb{F}_p^\ast$ there are $(p-1)$ choices for $\beta\in\mathbb{F}_p$ such that $\delta_{l,1} a_{n+1}\alpha  +a_n \alpha^2 -\delta_{n,1}a_ {n}^2 \alpha\neq a_n \beta$.
    Using \cref{rem:Closed-Open_Nott}, the subgroup $\langle x,y\rangle$ is open, as it contains $y$ with $D(y)=1$ and $[x,y]$ with $D([x,y])=n+2$, which satisfy the hypotheses of \cref{prop:qr}.

    \smallskip
    
    \noindent\textit{Case 2:} $n\equiv_p 0$. By \cref{lem:coeff-Comm-Nott}, we have 
        \[
            [x,y]\equiv t- a_n \alpha t^{n+2} + 2 a_n \left( \alpha^2 -\beta\right) t^{n+3} \pmod{\mathcal{N}_{n+3}}.
        \]
    By construction, $a_n$ and $\alpha$ are both non-zero, so $D([x,y])=n+1$. However, unlike the previous case, the depths of $y$ and $[x,y]$ do not satisfy the hypotheses of \cref{prop:qr}. We apply \cref{lem:coeff-Comm-Nott} again to compute
    \[
        [[x,y],y] \equiv t+ \alpha a_n( \alpha^2 -\beta) t^{n+4}\pmod{\mathcal{N}_{n+4}}.
    \]
    For each $\alpha\in\mathbb{F}_p^\ast$, there are $(p-1)$ choices for $\beta\in\mathbb{F}_p$ such that $\alpha^2 -\beta$ is non-zero. For these choices of $\alpha$ and $\beta$, by \cref{rem:Closed-Open_Nott}, the subgroup ${\langle x,y\rangle}$ is open, as it contains $y$ with $D(y)=1$ and $[[x,y],y]$ with $D([[x,y],y])=n+3$, which satisfy the hypotheses of \cref{prop:qr}.

    \smallskip
    
    \noindent\textit{Case 3:} $n\equiv_p -1$. As for the cases above, we compute commutators to find two depths which satisfy the hypotheses of \cref{prop:qr}. In this last case, we invoke \cref{lem:coeff-Comm-Nott} three times: 
        \begin{align*}
        [x,y] &\equiv t- 2 a_n \alpha t^{n+2} +  \left( 4 a_n \alpha^2 -3 a_n \beta -  \delta_{l,1}  a_{n+1} \alpha \right) t^{n+3} \pmod {\mathcal{N}_{n+3}},\\
        [[x,y],y] &\equiv t+ 2 a_n \alpha^2 t^{n+3} + 4 a_n\alpha\left( \beta-\alpha^2 \right) t^{n+4} \pmod{ \mathcal{N}_{n+4}}, \\
        [[[x,y],y],y]&\equiv t+ 2a_n \alpha^2\left( \beta - \alpha^2 \right) t^{n+5}\pmod{\mathcal{N}_{n+5}}.
    \end{align*}
    Since $a_n$ is non-zero, for each $\alpha\in\mathbb{F}_p^\ast$ there are $(p-1)$ choices for $\beta\in\mathbb{F}_p$ such that $\beta -\alpha^2$ is non-zero, and for these $D([[[x,y],y],y])=n+4$. Using \cref{rem:Closed-Open_Nott}, the subgroup $\langle x,y\rangle$ is open, as it contains {$y$} with $D(y)=1$ and $[[[x,y],y],y]$ with $D([[[x,y],y],y])=n+4$, which satisfy the hypotheses of \cref{prop:qr}.
    
    To conclude the proof, we observe that the claim about the measure of the openizer of $x$ follows by noting that for every $z\in\mathcal{N}_3$, $yz\equiv y \pmod{\mathcal{N}_{3}}$.
\end{proof}

\begin{proposition}\label{prop:Nott}
Let $p\ge 5$ be a prime. Then, for any non-trivial element $ x\in \mathcal{N}$ we have $$\mu(\mathbb{O}_x)\ge (p-1)^2\mu(\mathcal{N}_3)=\frac{(p-1)^2}{p^2}>\frac{1}{2}.$$
In particular, the graph  $\Dv(\mathcal{N})=\Gv(\mathcal{N})\setminus \{ 1\}$ is connected of diameter 2.
\end{proposition}
\begin{proof}
 Let $x\in \mathcal{N}$ with $D(x)=n$. If $n\not\equiv_p 1,0,-1$, then $x$ together with any $y\in \mathcal{N}$ with $D(y)=1$ generate an open subgroup in $\mathcal{N}$ by \cref{prop:qr}. So $\mu(\mathbb{O}_x)\ge 1-\frac{1}{p}$ in this case.
In the other cases, the lower bound for the measure of the openiser follows from \cref{lem:Nott10-1}. We conclude using \cref{lem:big_open}.
\end{proof}

%%%%%%%%%%%%%%%%%%%%%%%%%%%%%%%%%%%%%%%%%%%%%%%%%%%%%%%%%%

\section{The group \texorpdfstring{$C_p\wr \mathbb{Z}_p$}{CpwrZp}}\label{sec:WreathProcuts}

Consider the profinite wreath product $W=C_p\wr \Zp$.  We will use the isomorphism $W \cong \mathbb{F}_p[\![X]\!] \rtimes \langle t\rangle$ with $\mathbb{Z}_p =\langle t\rangle$ acting by multiplication by $1+X$. 
We denote by $B=\mathbb{F}_p[\![X]\!]$ the base group of the wreath product.
We start with a lemma.

\begin{lemma}
    For any $\lambda\in \Zp$, we have that $(f,t)^\lambda = \left( f\cdot \frac{(1+X)^\lambda -1}{X}, t^\lambda\right)$.
    \begin{proof}
First recall that the multiplication in $W$ works as follows:
\[
  (u, t^i) \cdot (v, t^j) = (u + t^i v, t^{i+j}).
\]
    We first prove the claim for $\lambda=n\in \mathbb{N}$ by induction; then, the claim follows by a standard convergence argument. 

    For $n=1$, $(f,t)^1 = (f,t) = \left( f\cdot \frac{1+X -1}{X}, t\right)$. For $n\ge 2$ 
    \begin{align*}
      (f,t)^{n+1} & = (f,t)^n \cdot (f,t) = \left(f\cdot \frac{(1+X)^n -1}{X}, t^n\right) \cdot (f, t) \\ &= \left(f\cdot \frac{(1+X)^n -1 + X(1+X)^n}{X}, t^{n+1}\right) = \left(f\cdot \frac{(1+X)^{n+1} -1}{X}, t^{n+1}\right).   
    \end{align*}

Now, take a sequence $n_i\to \lambda$ and we have that $(f,t)^{n_i} \to (f,t)^\lambda$.
    %    We work in the finite quotients $W_k = \mathbb{F}_p[X]/(X)^{p^k} \rtimes \langle 1+X \rangle$ and $\Zp/p^k\Zp \cong \mathbb{Z}/p^k \mathbb{Z}$.
     %   It is clear that $(1+X)^{p^k}=1$ in $B_k$. 
    \end{proof}
\end{lemma}

\begin{proposition}\label{prop:CpwrZp}
For the pro-$p$ group $W=C_p\wr \mathbb{Z}_p$, we have that $\Dv(W)=\Gv(W)\smallsetminus\{1\}$ and $\mathrm{diam}(\Dv(W)) =2$.
\end{proposition}

\newcommand{\psFp}{\mathbb{F}_p[\![X]\!]}
\begin{proof}
It can be easily seen that $$\mathbb{O}_X = \{(f, t^\lambda)\in W \mid \lambda \in \mathbb{Z}_p^\times\} \text{\quad and \quad}
 \mathbb{O}_t = \{(f, t^\lambda)\in W \mid f\neq 0\}.$$

For a non-identity element $g=(f,t^\lambda)\in W$ we will show that there are many elements in $N=\mathbb{O}_X \cap \mathbb{O}_t$ connected to $g$. Consider the element $h=(f',t)\in N$ and suppose that it is not connected to $g$, then we must have that $\langle g,h\rangle \cap B =\{0\}$. A direct computation shows that this can only happen if $g\in \langle h\rangle$, that is
\[
   f = f' \cdot \frac{(1+X)^\lambda -1}{X}.
\]
If we now consider two non-identity elements $g_1,g_2 \in W$, from the above condition it is clear that there exist $h\in N$ which is connected to both $g_1$ and $g_2$. Hence the proper virtually generating graph of $W$ is connected and of diameter $2$.
\end{proof}

\begin{remark}
 We can also estimate the size of openizers in $W$. In the previous proof we showed that for any $g=(f,t^\lambda)\in W$ the set 
 \[
    \mathbb{O}_X \cap \mathbb{O}_t \cap \left \{(f',t^\mu) \in W\ \big \vert\ f\neq f' \frac{(1+X)^\lambda -1}{X} \right\}
 \]
 is contained in the openizer of $g$. It is easy to see that $\mu(\mathbb{O}_X)=1-\frac{1}{p}$, $\mu(\mathbb{O}_t)=1$ and the above set has measure $1-\frac{1}{p}$.
\end{remark}

%%%%%%%%%%%%%%%%%%%%%%%%%%%%%%%%%%%%%%%%%%%%%%%%%%%%%%

\section{Small openizers}\label{sec:small}

Let $G$ be a pro-$p$ group. As we saw in \cref{lem:big_open}, if $g\in G$ and $\mathbb O_g\neq \emptyset,$
then $\mu(\mathbb O_g)>0.$ %Indeed, if $\mathbb O_g\neq \emptyset,$ then there exists $x\in G$ such that $\langle g,x\rangle$ is an open subgroup of $G$. Since the probability that $h\in H$ is such that $\langle g, h\rangle=H$ is $1-1/p,$ it follows that $$\mu(\mathbb O_g)\geq \left(1-\frac{1}{p}\right)\frac{1}{|G:H|}.$$ 
However, we are going to prove that there is no possible lower bound for the value of $\mu(\mathbb O_g).$

	\begin{lemma}\label{lem:p_elts_symmetric}
		Let $W_n$ be a Sylow $p$-subgroup of $S_{p^n}.$ Then the probability 
		that an element of $W_n$ has order $p^n$ is at most $(1-\frac{1}{p})^n$.
	\end{lemma}
	\begin{proof}
		We prove by induction on $n$ that the number $\gamma_n$ of elements 
		of $W_n$ of order $p^n$ is $(p-1)^np^{\frac{p^n-1}{p-1}-n}$. Note that $\lvert W_n\rvert=p^{\frac{p^n-1}{p-1}}$
		
		The base case $n=1$ is clear: $W_1=C_p$ has exactly $p-1$ elements of 
		order $p$, and the formula holds. 
		
		For the inductive step, since $W_n=C_p\wr W_{n-1},$ an element of 
		order $p^n$ in $W_n$ is of the form
		$x=(y_1,\dots,y_{p^{n-1}})w$ with $w\in W_{n-1}$ and 
		$(y_1,\dots,y_{p^{n-1}})\in C_p^{p^{n-1}}.$
		One verifies that $x$ has order $p^n$ if and only if $w$ has order 
		$p^{n-1}$ and $y_1\cdots y_{p^{n-1}}\neq 1.$
		By induction the number of choices for $w$ is 
		$\gamma_{n-1}$, while the number of tuples 
		$(y_1,\dots,y_{p^{n-1}})\in C_p^{p^{n-1}}$ with 
		$y_1\cdots y_{p^{n-1}}\neq 1$ is $(p-1)p^{{p^{n-1}}-1}$. %\footnote{there are $p^{p^{n-1}}$ vectors, exactly $p^{p^{n-1}}/p$ of the products $y_1\cdots y_{p^{n-1}}$ achieve a fixed $a\in C_p$, the tuples that have non-identity product are all but the ones that product to the identity.} 
        Thus
		$$
		\gamma_n=\gamma_{n-1}\cdot (p-1)p^{{p^{n-1}}-1}
		=(p-1)^{n-1}p^{\frac{p^{n-1}-1}{p-1}-(n-1)}\cdot (p-1)p^{p^{n-1}-1}=(p-1)^np^{\frac{p^n-1}{p-1}-n}.
		$$
		Since $|W_n|=p^{p^n-1},$ the probability that an element of $W_n$ 
		has order $p^n$ is
		$$
		\frac{\gamma_n}{|W_n|}=\frac{(p-1)^np^{\frac{p^n-1}{p-1}-n}}{p^{p^n-1}}= \frac{(p-1)^np^{\frac{p^n-1}{p-1}-n}}{p^{p^n-n+n-1}} =  (1-\frac{1}{p})^n p^{ \frac{p^n-1}{p-1}\frac{2-p}{p-1}} \le \left(1-\frac{1}{p}\right)^n. \qedhere
        $$
	\end{proof}

\comm{

Define the group 
\[
   H= \langle x,y \mid y^{-1}xy=x^{1-p},\ y^p=x \rangle, 
\]
this is a virtually procyclic pro-$p$ group with procyclic subgroup $C=\langle x \rangle \cong \mathbb{Z}_p$ and $[x,y]=x^{-p}$ generates an open subgroup of $H$.

}

\begin{proposition}\label{prop:small}
	For every $n\in \mathbb{N}$ there exists a group $G$ containing 
	an element $x$ such that $\mathbb{O}_x\neq \emptyset$ but 
	$\mu(\mathbb{O}_x)\leq (1-1/p)^n.$
\end{proposition}

\begin{proof}	
Let $\mathbb{Z}_p=\langle x \rangle$ 
and consider the wreath product $G=\mathbb{Z}_p\wr W_n\cong \mathbb{Z}_p^m \rtimes W_n$, 
where $m=p^n$. Assume $m>p$. 
Let $a=(1,0,\dots,0)\in \mathbb{Z}_p^m\leq G$. First, we claim that 
$\mathbb{O}_a\neq \emptyset.$ Indeed, consider an $m$-cycle $w$ in $W_n$, and let $K=\langle a, w\rangle.$
Then $K$ contains $a^{w^i}=(0,\dots,\underset{i}{1},\ldots,0)$
and, consequently, it contains $\mathbb{Z}_p^m$.
%$$u=[a,b^m]=([y,x],1,\dots,1)=(x^{p},1,\dots,1).$$
%It follows that $K$ contains 
%$\langle u, u^y, \dots, u^{y^{m-1}}\rangle \cong (p\mathbb{Z}_p)^m.$
Hence $K$ is open.

Now assume that $\tilde{b}=(h_1,\dots,h_m)\sigma \in \mathbb{O}_a.$
Assume by contradiction that $\sigma$ is not an $m$-cycle.
Write $\sigma=\tau_1\tau_2$ where $\tau_1$ is the cycle of $\sigma$ 
involving $1.$ Since $\tau_1$ is a $p$-element and is not an $m$-cycle, 
the length $u$ of $\tau_1$ is at most $m-p$. %Let $u$ be the length of $\tau_1$; w
Without loss of generality 
assume $\tau_1=(1\ 2\ \cdots\ u)$ and $\tau_2$ acts on 
$\{u+1,\dots,m\}.$
Let $\tilde{K}=\langle a, \tilde{b}\rangle.$ Since $a=(1,0,\dots,0)$ 
has support contained in $\{1,\dots,u\}$ and $\tau_1, \tau_2$ have 
disjoint supports, we have $\tilde{K} \leq K_1 \times K_2,$ where $K_1=\langle x,\, (h_1,\dots,h_u,0,\dots,0)\tau_1\rangle$ and 
$K_2=\langle (0,\dots,0,h_{u+1},\dots,h_m)\tau_2\rangle.$
Therefore, %recalling that $C=\langle x\rangle \le H$
$$
\tilde{K}\cap \mathbb{Z}_p^m \leq 
(K_1\cap \mathbb{Z}_p^u)\times (K_2\cap \mathbb{Z}_p^{m-u}).
$$
Since $K_2$ is pro-cyclic, $K_2\cap \mathbb{Z}_p^{m-u}$ has rank 
at most $1$, so $\tilde{K}\cap \mathbb{Z}_p^m$ has rank at most 
$u+1 \leq m-p+1 \leq m-1$ (for $m> p$). 
Hence $\tilde{K}\cap \mathbb{Z}_p^m$ is not open in $\mathbb{Z}_p^m,$ 
and therefore $\tilde{K}$ is not open in $G$, contradicting 
$\tilde{y}\in\mathbb{O}_a.$

 By the previous 
argument, if $\tilde{y}\in \mathbb{O}_a$ then $\sigma$ must be an 
$m$-cycle in $W_n.$ Therefore
$$
\mathbb{O}_a \subseteq \{(h_1,\dots,h_m)\sigma \in G : 
\sigma \text{ is an } m\text{-cycle in } W_n\}.
$$
Since the Haar measure on $G = \mathbb{Z}_p^m \rtimes W_n$ is the product of 
the Haar measure on $\mathbb{Z}_p^m$ and the uniform measure on $W_n,$ we have
$$
\mu(\mathbb{O}_a) \leq \mu(\{\tilde{y}\in G : 
\sigma \text{ is an } m\text{-cycle}\}) = 
\frac{|\{m\text{-cycles in } W_n\}|}{|W_n|} \le \left(1-\frac{1}{p}\right)^n,
$$
where the last equality follows from~\cref{lem:p_elts_symmetric}.
\end{proof}

Analyzing the properties of the virtually generating graph of
	$\mathbb{Z}_p \wr W_n$ is no straightforward task. We think it may be
	interesting to describe it in detail at least in the particular case
	$p = n = 2$, namely when $G = \mathbb{Z}_2 \wr D_4$
	(where $D_4$ denotes the dihedral group of degree 4). We identify $D_4$
	with the subgroup $\langle (1,2,3,4),(1,3)\rangle$ of $S_4$ and we write
	any element $g \in G$ in the form $g = v_g \sigma_g$ with
	$v_g \in V=\mathbb{Z}_2^4$ and $\sigma_g \in D_4$. For every $x \in D_4,$
	denote by $\Sigma_x$ the set of elements $g = v_g \sigma_g$ with
	$\sigma_g = x$. Since $\mu(\Sigma_x) = 1/8,$ for every $y \in G$ we may
	view the quantity
	$\nu_x(y) = 8 \mu(\mathbb{O}_y)$
	as the conditional probability that $\langle x, y \rangle$ is an open
	subgroup of $G,$ given that $x \in \Sigma_x.$

\begin{proposition}
\label{prop:Z2wrD4}
Let $g \in G=\mathbb Z_2 \wr D_4= V\rtimes D_4$.
\begin{enumerate}
	\item If $g=(a_1,a_2,a_3,a_4)\in V,$ then either $$\det\begin{pmatrix}
		a_1&a_2&a_3&a_4\\a_4&a_1&a_2&a_3\\a_3&a_4&a_1&a_2\\1&1&1&1\end{pmatrix}= 0,$$ in which case $g$ is an isolated vertex of $\Gv(G)$, or $\nu_{(1,2,3,4)^{\pm 1}}(g)=1.$ Moreover, $\mathbb O_g\cap \Sigma_x=\emptyset$ if $x\neq (1234)^{\pm 1}.$
	\item If either  $\sigma_g\notin \langle (13)(24)\rangle$ or $g=(u_1,u_2,u_3,u_4)(13)(24)$ with $u_1-u_3\neq u_2-u_4,$ then $\nu_{(1234)^{\pm 1}}(g)=1.$
	\item If $g=(u_1,u_2,u_3,u_4)(13)(24)$ and $u_1-u_3=u_2-u_4,$ then $\mathbb O_g\cap \Sigma_{(1,2,3,4)^{\pm 1}}=\emptyset$. Moreover,  $g$ is an isolated vertex of $\Gv(G)$ if $u_1=u_3$; $\nu_{(24)}(g)=1$ otherwise.
\end{enumerate}
It follows that $\Dv(G)$ is connected, and its diameter is equal to $3$.
\end{proposition}

\begin{proof} 1) Assume $g=(a_1,a_2,a_3,a_4).$ We have already noticed in the previous discussion that $\mathbb O_g \cap \Sigma_x=\emptyset$ if $x\neq (1234)^{\pm{1}}.$ Now let  $y=(x_1,x_2,x_3,x_4)(1234)\in \Sigma_{(1234)}$. Then
$$\begin{aligned}\langle g, y\rangle&= \langle g, g^y, g^{y^2}, g^{y^3}, y^4\rangle\\&=\langle (a_1,a_2,a_3,a_4), (a_2,a_3,a_4,a_1), (a_3,a_4,a_1,a_2), (a_1,a_2,a_3,a_4), (z,z,z,z)\rangle,
\end{aligned}$$ with
$z=x_1+x_2+x_3+x_4.$ If 
$$\det\begin{pmatrix}
	a_1&a_2&a_3&a_4\\a_4&a_1&a_2&a_3\\a_3&a_4&a_1&a_2\\a_2&a_3&a_4&a_1\end{pmatrix}\neq 0,$$ then
	$\mathbb O_g=\Sigma_{(1,2,3,4)}\cup \Sigma_{(1,4,3,2)}.$
Otherwise $g$ is not isolated if and only if $$\det\begin{pmatrix}
	a_1&a_2&a_3&a_4\\a_4&a_1&a_2&a_3\\a_3&a_4&a_1&a_2\\1&1&1&1\end{pmatrix}\neq 0,$$
and in this case $\mathbb O_g \cap \Sigma_{(1,2,3,4)}=\{ (x_1,x_2,x_3,x_4)(1,2,3,4)\mid x_1+x_2+x_3+x_4\neq 0\},$ and consequently $\nu_{(1234)}(g)=1.$

\noindent
2) Up to conjugacy, we may assume that $g$ is conjugate in $G$ to one of the elements  listed in the first column of Table 1. Given $y=(x_1,x_2,x_3,x_4)(1234)\in \Sigma_{(1234)}$, consider the elements $w=(z_1,z_2,z_3,z_4)\in \langle g, y \rangle \cap \mathbb Z_2^4$ described in the second column of Table \ref{tab:gw}.
\begin{table}[h]
	\centering
	\begin{tabular}{|c|l|}
		\hline
		\textbf{$g$} & \textbf{$w$} \\
		\hline
		$(a, b, 0, d)(13)$ & 
		$\begin{aligned}
			(ygy)^2 = &\bigl(x_1+x_2+x_3+x_4+b+d,\, 2(x_1+x_2),\\
			&\phantom{\bigl(}x_1+x_2+x_3+x_4+b+d,\, 2(x_3+x_4+a)\bigr)
		\end{aligned}$ \\
		\hline
		$(a, 0, b, 0)(12)(34)$ &
			$(gy)^2= \bigl(a+b+x_2+x_4, 2x_1, a+b+x_2+x_4, 2x_3)
	$ \\
		\hline
		$(a,b,0,0)(13)(24)$ &
			$gy^2= \bigl(a+x_3+x_4, b+x_1+x_4, x_1+x_2, x_2+x_3)
		$\\
		\hline
		$(a,0,0,0)(1234)$ &$yg^{-1}=(x_1-a,x_2,x_3,x_4)$\\
		\hline
	\end{tabular}
	\caption{$g$ and $w$}
	\label{tab:gw}
\end{table}
Since $w, w^g, w^{g^2}, w^{g^3}\in \langle g, y\rangle \cap \mathbb Z_2^4,$ it follows that $\langle g, y\rangle$ is open if
$$\begin{aligned}f_g(x_1,x_2,x_3,x_4)&=\det\begin{pmatrix}
z_1&z_2&z_3&z_4\\z_2&z_3&z_4&z_1\\z_3&z_4&z_1&z_2\\z_4&z_1&z_2&z_3
\end{pmatrix}\\[0.5em]&=(z_1+z_2+z_3+z_4)(z_1-z_2+z_3-z_4)((z_1-z_3)^2+(z_2-z_4)^2)\neq 0.\end{aligned}$$
Except when $g=(a,b,0,0)(13)(24)$ and $a=b,$ $f_g(x_1,x_2,x_3,x_4)$
is a non zero polynomial in $\mathbb Z_2[x_1,x_2,x_3,x_4].$ In particular
$\{(x_1,x_2,x_3,x_4)\in \mathbb Z_2^4\mid f_g(x_1,x_2,x_3,x_4)=0\}$ is a subset of $\mathbb Z_2^4$ with measure $0,$ and consequently $\nu_{(1234)}(g)=1.$

\noindent 3) Assume $g=(u_1,u_2,u_3,u_4)(13)(24)$ and $u_1-u_3=u_2-u_4.$
Up to conjugacy we may assume $g=(a,a,0,0)(13)(24).$
Consider $y=(x_1,x_2,x_3,x_4)(2,4).$ Then
$$\langle (gy)^2, ((gy)^2)^y, g^2, y^2\rangle \leq \langle g, y\rangle \cap \mathbb Z_2^4$$ and 
$$\begin{aligned}
\det\begin{pmatrix}(gy)^2\\((gy)^2)^y\\g^2\\y^2\end{pmatrix}&=
\det\begin{pmatrix}
	a+x_1+x_3 & 2(a+x_4)  & a+x_1+x_3 & 2x_2      \\
	a+x_1+x_3 & 2x_2      & a+x_1+x_3 & 2(a+x_4)      \\
	a          & a         & a          & a          \\
	2x_1       & x_2+x_4  & 2x_3       & x_2+x_4
\end{pmatrix}\\[1em]&=
8a(x_3-x_1)(a+x_4-x_2)(x_1+x_3-x_2-x_4)
.
\end{aligned}$$
We conclude that if $a\neq 0$ then $\nu_{(24)}(g)=1.$ 

Now we claim that if $y = (x_1, x_2, x_3, x_4)\sigma$, where $\sigma$ 
is a $4$-cycle, then $g$ and $y$ are not adjacent. It is not restrictive 
to assume $\sigma = (1234)$. Let $K = \langle g, y \rangle = 
\langle gy^{2}, g \rangle$. Notice that, setting $W := \{(y_1,y_2,y_3,y_4) \in \mathbb{Z}_2^4 \mid y_1+y_3 = y_2+y_4\}$, we have that
\[
gy^{2} = (a+x_3+x_4,\, a+x_1+x_4,\, x_1+x_2,\, x_2+x_3) 
\in W.
\]
In particular, since $y$ normalizes $W$ and $y^4 = (z,z,z,z) \in W$, for $z=x_1+x_2+x_3+x_4$,
it follows that $K \leq V\langle y \rangle$ and
\[
K \cap \mathbb{Z}_2^4 \leq V\langle y \rangle \cap \mathbb{Z}_2^4
= W\langle y^4 \rangle = W \cong \mathbb{Z}_2^3.
\]
Hence $K$ is not an open subgroup of $G$.

It remains to prove that $g=(13)(24)$ is an isolated vertex of $\Gv(G).$ Suppose, by contradiction, that $y=(x_1,x_2,x_3,x_4)\sigma_y \in \mathbb O_g.$ We already know that $\mathbb O_g\cap \Sigma_{(1234)^{\pm 1}}=\mathbb O_g \cap \Sigma_{\rm{id}}=\emptyset.$ Moreover, for every $x\in  D_4,$ $\mathbb O_g\cap \Sigma_y=\mathbb O_g\cap \Sigma_{gy},$ so we may assume $\sigma_y \in \{(13),(13)(24)\}.$ If $y=(x_1,x_2,x_3,x_4)(13)$ then $\langle g, y \rangle \cap \mathbb Z_2^4\leq \langle (x_1,x_2,x_3,x_4),(x_3,x_2,x_1,x_4),(x_1,x_4,x_3,x_2)\rangle$, which is not open. Finally assume $y = (x_1, x_2, x_3, x_4)(12)(34)$.  Notice that
$W := \{(t,u,-t,-u) \in \mathbb{Z}_2^4 \mid t, u \in \mathbb Z_2\}$ is normalized by $K=\langle g,y\rangle$ and contains $K^\prime.$
In particular $K\cap \mathbb Z_2^4 = W+\langle y^2\rangle$ is not open in $\mathbb Z_2^4.$

We can now give a complete description of the graph $\Gv(G)$. The isolated 
vertices either belong to $\mathbb{Z}_2^4$ or are of the form 
$(u_1,u_2,u_3,u_4)(13)(24)$ with $u_1=u_3$ and $u_2=u_4$. Let $\Omega_1$ be 
the set of non-isolated vertices of the form $(u_1,u_2,u_3,u_4)\sigma$ with 
$\sigma\neq(13)(24)$. By (1) and (2), any two elements of $\Omega_1$ have a common neighbour in $\Sigma_{(1234)}$, hence $\Omega_1$ induces 
a subgraph of diameter at most $2$. Now let $\Omega_2$ be the set of the vertices of the form 
$(u_1,u_2,u_3,u_4)(13)(24)$ with $u_1-u_3=u_2-u_4$ and $u_1\neq u_3.$ 
By (3), any two elements of $\Omega_2$ have a common neighbour in $\Sigma_{(24)}$, so their distance in $\Gv(G)$ is at most 2. Finally assume $g_1\in \Omega_1$ and $g_2\in \Omega_2.$ There exists $y\in \mathbb O_{g_2}\cap \Sigma_{(12)}$ and the distance between $y$ and $g_1$ is at most 2, so the distance between $g_1$ and $g_2$ is at most 3.
In the particular case when $g_1\in V,$ since $\mathbb O_{g_1}\subseteq \Sigma_{(1234)^{\pm 1}}$ and $\mathbb O_{g_2}\cap \Sigma_{(1234)^{\pm 1}}=\emptyset,$
the distance between $g_1$ and $g_2$ is exactly 3. We conclude that
$\Dv(G)$ is connected, with diameter 3.
\end{proof}

\bibliography{Mylibrary} 
\bibliographystyle{abbrv}

\end{document}